\documentclass[12pt,leqno]{amsart}
\usepackage{amsfonts,amsthm,amssymb,amsmath,amscd}
\usepackage{mathrsfs}

\def\N{{\mathbb N}}
\def\Z{{\mathbb Z}}

\def\R{{\mathbb R}}
\def\C{{\mathbb C}}
\def\H{{\mathfrak H}}

\def\sq{\hbox{\rlap{$\sqcap$}$\sqcup$}}
\def\qed{\ifmmode\sq\else{\unskip\nobreak\hfil
         \penalty50\hskip1em\null\nobreak\hfil\sq
         \parfillskip=0pt\finalhyphendemerits=0\endgraf}\fi}

\def\M{{\mathcal M}}
\def\N{{\mathcal N}}

\def\T{{\mathcal T}}

\newtheorem{theorem}{Theorem}
\newtheorem{lemma}[theorem]{Lemma}
\newtheorem{prop}[theorem]{Proposition}
\newtheorem{cor}[theorem]{Corollary}

\numberwithin{theorem}{section}
\numberwithin{equation}{section}

\allowdisplaybreaks[1]

\title[Rankin-Selberg method for Jacobi forms]{
Rankin-Selberg method for Jacobi forms of integral weight
and of half-integral weight on symplectic groups.
}
\author[S.~Hayashida]{Shuichi Hayashida
}
\date{\today}

\begin{document}

\begin{abstract}
In this article we show analytic properties of
certain Rankin-Selberg type Dirichlet series for holomorphic Jacobi cusp forms
of integral weight and of half-integral weight.
The numerators of these Dirichlet series are the inner products of Fourier-Jacobi coefficients
of two Jacobi cusp forms.
The denominators and the range of summation of these Dirichlet series
are like the ones of the Koecher-Maass series.
The meromorphic continuations and functional equations of these Dirichlet series
are obtained.
Moreover,
an identity between the Petersson norms of Jacobi forms 
with respect to linear isomorphism between Jacobi forms of integral weight and
half-integral weight is also obtained.
\end{abstract}

\maketitle

\section{Introduction}
\subsection{}
The aim of this paper is to show analytic properties of certain Rankin-Selberg type
Dirichlet series of Jacobi cusp forms.
In~\cite{KoSk} Kohnen and Skoruppa introduced the following Dirichlet series
associated to Siegel cusp forms of degree two: 
\begin{eqnarray*}
  \sum_{m=1}^\infty \frac{\langle f_m, g_m \rangle}{m^s},
\end{eqnarray*}
where $f_m$ and $g_m$ are $m$-th Fourier-Jacobi coefficients of 
two Siegel cusp forms of degree two $F$ and $G$, respectively,
and where $\langle f_m, g_m \rangle$ is the Petersson inner product of Jacobi cusp forms
$f_m$ and $g_m$.
They showed a meromorphic continuation and the functional equation of this Dirichlet series.
This result has been generalized to Siegel cusp forms of arbitral degree $n$ by Yamazaki~\cite{Ya3}.
It means that $F$ and $G$ can be replaced by Siegel cusp forms of arbitral degree $n$
and $f_m$ and $g_m$ are Fourier-Jacobi coefficients of $F$ and $G$ of
not only integer index, but also matrix index $m$.
If $m$ runs over matrices of a fixed size, the range of summation
and the denominators of Yamazaki Dirichlet series are the same to the ones
of the Koecher-Maass Dirichlet series.

In this paper we start with two Jacobi cusp forms $\phi_\M$ and $\psi_\M$
instead of two Siegel cusp forms $F$ and $G$.
Here $\phi_\M$ and $\psi_\M$ are Jacobi cusp forms of index $\M$ and
where $\M$ is a half-integral symmetric matrix. 
The Fourier-Jacobi coefficients $\phi_\N$ and $\psi_\N$ of
$\phi_\M$ and $\psi_\M$, respectively, are also Jacobi cusp forms.
We will show a meromorphic continuation and the functional equation of
similar Dirichlet series as above.

The importance of this result is that one can apply it to obtain also similar result
for Jacobi cusp forms (which includes Siegel cusp forms)
of \textit{half-integral weight} in a generalized plus space.
This result is a generalization of the result for the Rankin-Selberg type Dirichlet series
of Siegel cusp forms of half-integral weight obtained by Katsurada and Kawamura in
~\cite[Corollary to Proposition 3.1]{KaKa}. 
It means that in~\cite{KaKa} they treated Fourier coefficients of Siegel cusp forms of
half-integral weight to construct the Rankin-Selberg type Dirichlet series.
In this paper
we shall treat \textit{Jacobi cusp forms of half-integral weight of
arbitral degree} instead of Siegel cusp forms of half-integral weight
and we also shall treat \textit{Fourier-Jacobi coefficients}
instead of Fourier coefficients.

\subsection{}
To be more precise,
let $L_n^*$ be the set of all half-integral symmetric matrices of size $n$
and $L_n^+$ be the subset of all positive-definite matrices in $L_n^*$.

We fix a matrix $\M \in L_{r}^+$.
We set
\begin{eqnarray*}
  L_{t,r}^+(\M) &:=& \left\{  \mathcal{N} = \begin{pmatrix} N & \frac12 R \\ \frac12 {^t R} & \M \end{pmatrix} 
   \in L_{t+r}^+ \, | \,  N \in  L_t^+ , R \in \Z^{(t,r)}\right\}
\end{eqnarray*}
and we put
\begin{eqnarray*}
  B_{t,r}(\Z) &:=&
  \left\{
    \gamma = \begin{pmatrix} \gamma_1 & \gamma_2 \\ 0 & 1_r  \end{pmatrix}
    \in GL(t+r,\Z)
    \, | \,
    \gamma_1 \in GL_t(\Z),  \gamma_2 \in \Z^{(t,r)}
  \right\} .
\end{eqnarray*}
The group $B_{t,r}(\Z)$ acts on $L_{t,r}^+(\M)$ by
$\gamma \cdot \N := \N[^t \gamma]$
for $\gamma \in B_{t,r}(\Z)$ and $\N \in L_{t,r}^+(\M)$.
Here we put $X[Y] := {^t Y} X Y$ for matrices $X$ and $Y$ of suitable sizes.

We set
\begin{eqnarray*}
    \epsilon(\M)  
    &:=& \{
   \gamma \in GL(r,\Z) \, | \, 
  \gamma \cdot \M = \M \}.
\end{eqnarray*}
For $\mathcal{N} \in L_{t,r}^+(\M)$ we set
\begin{eqnarray*}
  \epsilon_{t,r}(\mathcal{N})
  &:=& \{
   \gamma \in B_{t,r}(\Z) \, | \, 
   \gamma \cdot \N = \N \}.
\end{eqnarray*}

The symbol $\mbox{Sp}(n,\R)$ denotes the real symplectic group of size $2n$.
We put $\Gamma_n$ := $\mbox{Sp}(n,\Z) = \mbox{Sp}(n,\R) \cap \Z^{(2n,2n)}$.
The symbol $\H_n$ denotes the Siegel upper half space of size $n$.

Let $\phi_{\M}$ and $\psi_{\M}$ be Jacobi cusp forms of weight $k$ of index $\M$
on $\Gamma_n$
(cf.~\cite[Definition 1.3]{Zi}).
In the case $r = 0$ we regard $\phi_\M$ and $\psi_\M$  as Siegel cusp forms of weight $k$
with respect to $\Gamma_n$.
Let $0 < t \leq n$ and
we take Fourier-Jacobi expansions of $\phi_{\M}$ and $\psi_{\M}$:
\begin{equation}\label{id:theta_decom_phi_psi}
\begin{split}
  \phi_{\M}(\tau,z) e(\M \omega)
  &=
  \sum_{\N \in L_{t,r}^+(\M)} \phi_{\N}(\tau',z') e(\N \omega'), \\
  \psi_{\M}(\tau,z) e(\M \omega)
  &=
  \sum_{\N \in L_{t,r}^+(\M)} \psi_{\N}(\tau',z') e(\N \omega'),
\end{split}
\end{equation}
where
$\begin{pmatrix} \tau & z \\ ^t z & \omega \end{pmatrix}
= \begin{pmatrix} \tau' & z' \\ ^t z' & \omega' \end{pmatrix}$,
$\tau \in \H_n$, $\omega \in \H_r$, $z \in \C^{(n,r)}$,
$\tau' \in \H_{n-t}$, $\omega' \in \H_{t+r}$ and $z' \in \C^{(n-t,t+r)}$.
We have the fact that $\phi_\N$ and $\psi_\N$ are Jacobi cusp forms of
weight $k$ of index $\N$ on $\Gamma_{n-t}$.

The aim of this article is to obtain analytic properties of the Dirichlet series
\begin{eqnarray}\label{def:dirichlet}
    D_t(\phi_{\M}, \psi_{\M} ; s)
   & := &
    \sum_{\mathcal{N} 
    \in B_{t,r}(\Z) \backslash L_{t,r}^+(\M) 
    } \frac{\langle  \phi_{\mathcal{N}}, \psi_{\mathcal{N}} \rangle}{
    |\epsilon_{t,r}(\mathcal{N})| \det(\mathcal{N})^{s}},
\end{eqnarray}
where we denote by $\langle \phi_\N, \psi_\N \rangle$ the Petersson inner product
of $\phi_\N$ and $\psi_\N$ (see \S~\ref{s:jacobi_int}).
This Dirichlet series converges for sufficiently large $\mbox{Re}(s)$
(see Lemma~\ref{lem:Diri_conv}).
If $t=n$, then $\phi_\N$ and $\psi_\N$ are Fourier coefficients of $\phi_\M$
and $\psi_\M$, respectively,
and we set $\langle \phi_\N, \psi_\N \rangle = \phi_\N  \overline{\psi_\N}$
in this case.
We remark
$
 \displaystyle{ D_1(\phi_\M,\psi_\M) =
   \frac12
       \sum_{\mathcal{N} \in \Z_{>0} 
    } \frac{\langle  \phi_{\mathcal{N}}, \psi_{\mathcal{N}} \rangle}{
    \mathcal{N}^{s}}}
$ if $t = 1$ and $r = 0$.

Remark that if $r = 0$ and if $n \geq 2$,
then $\phi_{\M}$ and $\psi_{\M}$ are Siegel cusp forms.
And analytic properties
of the Dirichlet series $D_t(\phi_{\M}, \psi_{\M} ; s)$ in the case $r=0$
have been shown by Maass~\cite{HM73} (for $t = n = 2$),
by Kohnen-Skoruppa \cite[Theorem 1]{KoSk} (for $t=1$ and $n=2$),
by Kalinin \cite{kal84} (for $t = n \geq 2$)
and by Yamazaki \cite{Ya3} (for $t \geq 1$ and $n \geq 2$).
Moreover, if $r=1$, then meromorphic continuations, functional equations and residues of 
the above Dirichlet series have been shown in the papers by
Kohnen-Zagier~\cite[p.189-191]{KoZa} (for $t = n = 1$ and $\M = 1$),
by Katsurada-Kawamura\cite[Proposition 3.1]{KaKa} (for $t = n \geq 2$ and $\M = 1$)
and by Imamo\=glu-Martin\cite[Theorem 2(d) and Proposition 1(b)]{ImaMa} (for $t = n = 1$ and for arbitral integer $\M$).

The main result in the present article is that the Dirichlet series
$D_t(\phi_{\M}, \psi_{\M} ; s)$ (for $1 \leq t \leq n$ and for arbitral index $\M$)
has a meromorphic continuation
to the whole complex plane and has  a functional equation (see Theorem~\ref{th:D_t_M}).
The residue at $s =  k - \frac{r}{2}$ is also determined   
(cf. Theorem~\ref{th:D_t_M}).
Such properties are shown by using Rankin-Selberg method. 
The properties of certain real analytic Siegel-Eisenstein series,
which are necessarily to 
prove Theorem~\ref{th:D_t_M},
have been shown by Kalinin~\cite{kalinin} (for $t=n$) and by Yamazaki~\cite{Ya3}
(for $1\leq t < n$).
Hence the main issue in this paper is to obtain an integral expression of
the Dirichlet series $D_t(\phi_{\M}, \psi_{\M} ; s)$ by using such Siegel-Eisenstein series.
To obtain the integral expression, we shall refine a method treated by Katsurada and Kawamura
in~\cite[Proposition 3.1]{KaKa}.
It means that we take vector valued modular forms through theta decompositions
of two Jacobi forms and take the inner product of these two vector valued modular forms. 
By coupling this inner product with a certain Siegel-Eisenstein series
we obtain the integral expression of the Dirichlet series $D_t(\phi_{\M}, \psi_{\M} ; s)$
(cf. Proposition~\ref{prop:int_D}).
To show the integral expression of the Dirichlet series,
we use the compatibility between the theta decomposition and
the Fourier-Jacobi expansion of Jacobi forms.

\subsection{}
As for half-integral weight case, we will show also analytic properties
for similar Dirichlet series of certain Jacobi cusp forms (including certain Siegel cusp forms)
of {\it half-integral weight} of \textit{certain indices}.
Such Jacobi forms of half-integral weight belong to the so-called plus-space,
which is a generalization of Kohnen plus-space
of elliptic modular forms of half-integral weight.
A generalization of Kohnen plus-space for Jacobi forms has been introduced in~\cite{matrix_integer}.
Let $\M = \begin{pmatrix} N & \frac12 R \\ \frac12 {^t R} & 1 \end{pmatrix}$ be a matrix
in $L_{r-1,1}^+(1)$.
We put $\mathfrak{M} = 4N - R {^t R}$.
We assume that $k$ is an even integer.
It is shown in~\cite{Ib} (for $r=1$) and in ~\cite{matrix_integer} (for $r > 1$)
that there exists a linear isomorphism
between the space $J_{k,\M}^{(n)}$ of Jacobi forms of weight $k$ of index $\M$ on $\Gamma_n$
and the space of Jacobi forms $J_{k-\frac12,\mathfrak{M}}^{(n)+}$,
where $J_{k-\frac12,\mathfrak{M}}^{(n)+}$ is a certain subspace of
Jacobi forms of weight $k-\frac12$ of index $\mathfrak{M}$ on $\Gamma_n$ (see \S\ref{s:plus_space_jacobi_forms}).
By this linear isomorphism the Fourier coefficients
of forms in  $J_{k,\M}^{(n)}$ and $J_{k-\frac12,\mathfrak{M}}^{(n)+}$ corresponds each other.

We remark that if $r=1$, then $\mathfrak{M} = \emptyset$ and
$J_{k-\frac12,\mathfrak{M}}^{(n)+}$ is the plus-space
of Siegel modular forms of weight $k-\frac12$ introduced by
Kohnen~\cite{Ko} (for $n=1$) and by Ibukiyama~\cite{Ib} (for $n > 1$).
Let $\phi_\mathfrak{M}$ and $\psi_\mathfrak{M}$ be Jacobi cusp forms in
$J_{k-\frac12,\mathfrak{M}}^{(n)+}$. We construct a Dirichlet series
$D_t(\phi_{\mathfrak{M}}, \psi_{\mathfrak{M}} ; s)$ in the same manner
as in the case of integral weight.
Then, by using the linear isomorphism between $J_{k,\M}^{(n)}$
and $J_{k-\frac12,\mathfrak{M}}^{(n)+}$,
we obtain a meromorphic continuation, a functional equation and residues
of $D_t(\phi_{\mathfrak{M}}, \psi_{\mathfrak{M}} ; s)$
(see Theorem~\ref{th:D_t_M_half}).

\ \\

\noindent
Acknowledgement:

This work was supported by JSPS KAKENHI Grant Number 80597766.

\section{Dirichlet series 
of Jacobi forms of integral weight}\label{s:jacobi_int}

We denote by $\H_n$ the Siegel upper half space of size $n$.
For any ring $R$, we denote by $R^{(l,m)}$ the set of all matrices
of size $l \times m$ with the entries in $R$.
The symbol $0^{(l,m)}$ denotes the zero matrix in $\C^{(l,m)}$.
We denote by $\delta_{i,j}$ the Kronecker delta.
It means that $\delta_{i,j} = 1$ if $i = j$ and 0 otherwise.
By abuse of language we put $\det(\M) = \det(4\M) = 1$, if the size of the matrix
$\M$ is $0$.

Let $k$ be an integer and let $\M \in L_r^+$.
We denote by $J_{k,\M}^{(n)}$ (resp. $J_{k,\M}^{(n)\, cusp}$) the space of Jacobi forms
(resp. Jacobi cusp forms)
of weight $k$ of index $\M$ on $\Gamma_n$.
(See the definition~\cite[Definition 1.3]{Zi}).

For $\phi$, $\psi$ $\in$ $J_{k,\M}^{(n)\, cusp}$, 
the Petersson inner product is defined by
\begin{eqnarray*}
  \langle \phi, \psi \rangle
  &:=&
  \int_{\mathcal{F}_{n,r}}
    \phi(\tau,z) \overline{ \psi(\tau,z)} e^{-4\pi Tr( \M v^{-1}[y])} \det(v)^{k - n - r - 1 } 
    \, du\, dv\, dx\, dy,
\end{eqnarray*}
where $\mathcal{F}_{n,r} := \Gamma_{n,r}^J \backslash (\H_n \times \C^{(n,r)}) $,
$\tau = u + i v$, $z = x + i y$,
$du = \prod_{i \leq j} u_{i,j}$,  $d v = \prod_{i \leq j} v_{i,j}$,
$dx = \prod_{i,j} x_{i,j}$ and $dy = \prod_{i,j} y_{i,j}$.
Here we put
\begin{eqnarray*}
  \Gamma_{n,r}^J
&:=&
  \left. \left\{
  \begin{pmatrix}
    A & 0 & B  & * \\
    * & 1_r & * & * \\
    C & 0 & D & * \\
    0 & 0 & 0 & 1_r
  \end{pmatrix}
  \in \Gamma_{n+r}
  \, \right| \,
  \begin{pmatrix}
    A & B \\ C & D 
  \end{pmatrix}
  \in \Gamma_n
  \right\} ,
\end{eqnarray*}
and the group $\Gamma_{n,r}^J$ acts on $\H_n \times \C^{(n,r)}$
in the usual manner (cf. \cite[p.193]{Zi}).
For the sake of simplicity we put $J_{k,\M}^{(0)\, cusp} := \C$ and for
$\phi$, $\psi$ $\in$ $\C$,
we set $\langle \phi, \psi \rangle := \phi \overline{\psi}$.

\begin{lemma}\label{lem:FJc_norm}
Let $\phi_\N$ be the $\N$-th Fourier-Jacobi coefficient of $\phi_\M$ defined in~(\ref{id:theta_decom_phi_psi}).
Then, there exists a constant $C'$ which does not depend on the choice of $\N$ such that
\begin{eqnarray*}
   |\phi_\N(\tau',z') e(i \N v'^{-1}[y']) (\det v')^{\frac{k}{2}} | 
  &<&
    C'
     \det(\N)^{\frac{k}{2}}
\end{eqnarray*}
for any $(\tau',z') \in \H_{n-t} \times \C^{(n-t,t)}$,
and where $v' = \mbox{Im}(\tau')$ and $y' = \mbox{Im}(z')$.
\end{lemma}
\begin{proof}

Since $\phi_\M$ is a Jacobi cusp form, there exists a constant $C_\phi$
which depends only on $\phi_\M$ such that
\begin{eqnarray*}
  |\phi_\M(\tau,z) \det(v)^{\frac{k}{2}} exp(-2\pi \, \mbox{Tr}(\M v^{-1}[y])) | < C_\phi
\end{eqnarray*}
for any $(\tau,z) \in \H_n \times \C^{(n,r)}$,
and where $v = \mbox{Im}(\tau)$ and $y = \mbox{Im}(z)$.
On the other hand, we have
\begin{eqnarray*}
  &&
  \phi_\N(\tau',z') e(i \N T') \\
  &=& 
  \int_{Sym_r(\Z) \backslash Sym_r(\R)}
  \int_{\R^{(t,r)}}
  \int_{Sym_t(\Z) \backslash Sym_t(\R)}
    \phi_\M\left( \tau + \left(\begin{smallmatrix} 0 & 0 \\ 0 & X'_1 \end{smallmatrix} \right),
     z + \left( \begin{smallmatrix} 0 \\  X'_3 \end{smallmatrix} \right) \right) \\
     &&
     \times
     e(\M(\omega+ X'_2))
    e(-\N \left(\begin{smallmatrix} X'_1 & X'_3 \\ ^t X'_3 & X'_2 \end{smallmatrix} \right) )
    \, dX'_1 \, dX'_3 \, dX'_2, 
\end{eqnarray*}
where $\left(\begin{smallmatrix} \tau & z \\ ^t z & \omega \end{smallmatrix} \right)  
= \left(\begin{smallmatrix} \tau' & z' \\ ^t z' & \omega' \end{smallmatrix} \right)  \in \H_{n+r}$,
$\tau \in \H_n$, $\omega \in \H_r$, $z \in M_{n,r}(\C)$,
$\tau' \in \H_{n-t}$, $\omega' \in \H_{t+r}$,
and $z' \in M_{n-t,t+r}(\C)$.

We decompose the matrix
$\left(\begin{smallmatrix} \tau & z \\ ^t z & \omega \end{smallmatrix} \right)  $ as
\begin{eqnarray*}
\left(\begin{matrix} \tau & z \\ ^t z & \omega \end{matrix} \right)  
=
\left(\begin{matrix} \tau' & z'_1 & z'_2 \\ ^t z'_1 & \tau_2 & z_2 \\ ^t z'_2 & ^t z_2 & \omega
\end{matrix} \right)  
=
 \left(\begin{matrix} \tau' & z' \\ ^t z' & \omega' \end{matrix} \right)  \in \H_{n+r}
\end{eqnarray*}
with
\begin{eqnarray*}
 \tau = \begin{pmatrix} \tau' & z'_1 \\ ^t z'_1 & \tau_2 \end{pmatrix},
 z = \left( \begin{matrix} z'_2 \\ z_2 \end{matrix} \right),
 z'  =\left( \begin{matrix} z'_1 &  z'_2 \end{matrix} \right),
 \omega' = \begin{pmatrix} \tau_2 & z_2 \\ ^t z_2 & \omega \end{pmatrix}, \\
 z'_1 \in \C^{(n-t,t)},
 z'_2 \in \C^{(n-t,r)},
 z_2 \in \C^{(t,r)},
 \mbox{ and }
 \tau_2 \in \H_t.
\end{eqnarray*}

We write
\begin{eqnarray*}
\left(\begin{matrix} v' & y'_1 & y'_2 \\ ^t y'_1 & v_2 & y_2 \\ ^t y'_2 & ^t y_2 & T
\end{matrix} \right)
&:=&
 \mbox{Im}\left(\begin{matrix} \tau' & z'_1 & z'_2 \\ ^t z'_1 & \tau_2 & z_2 \\ ^t z'_2 & ^t z_2 & \omega
\end{matrix} \right) ,
\end{eqnarray*}
$v' = \mbox{Im}(\tau')$, $v_2 = \mbox{Im}(\tau_2)$, $T = \mbox{Im}(\omega)$
and
$\displaystyle{
  T'
  :=
  \mbox{Im}(\omega')
  =  
  \left(\begin{matrix}   v_2 & y_2 \\  ^t y_2 & T \end{matrix} \right)}
$.

Furthermore, we write
$y' := \mbox{Im}(z') = 
 \mbox{Im}\begin{pmatrix} z'_1 & z'_2\end{pmatrix} = \begin{pmatrix} y'_1 & y'_2\end{pmatrix}$,
then we have
\begin{eqnarray*}
 &&
  |\phi_\N(\tau',z') e(i \N v'^{-1}[y']) (\det v')^{\frac{k}{2}} | \\
  &=&
  \left|
    e(-i\N T')
      \int_{Sym_r(\Z) \backslash Sym_r(\R)}
  \int_{\R^{(t,r)}}
  \int_{Sym_t(\Z) \backslash Sym_t(\R)}
    \phi_\M\left( \tau + \left(\begin{smallmatrix} 0 & 0 \\ 0 & X'_1 \end{smallmatrix} \right),
     z + \left( \begin{smallmatrix} 0 \\  X'_3 \end{smallmatrix} \right) \right)
          \right.
     \\
     &&
     \left.
     \times
     e(\M(\omega+ X'_2))
    e(-\N \left(\begin{smallmatrix} X'_1 & X'_3 \\ ^t X'_3 & X'_2 \end{smallmatrix} \right) )
    \, dX'_1 \, dX'_3 \, dX'_2
    \ e(i \N v'^{-1}[y']) (\det v')^{\frac{k}{2}}
    \right| \\
  &=&
  \left|
    e(-i\N T')
      \int_{Sym_r(\Z) \backslash Sym_r(\R)}
  \int_{\R^{(t,r)}}
  \int_{Sym_t(\Z) \backslash Sym_t(\R)}
    \phi_\M\left( \tau + \left(\begin{smallmatrix} 0 & 0 \\ 0 & X'_1 \end{smallmatrix} \right),
     z + \left( \begin{smallmatrix} 0 \\  X'_3 \end{smallmatrix} \right) \right)
          \right.
     \\
     &&
     \left.
     \times
     \det(\mbox{Im}(\tau + \left(\begin{smallmatrix} 0 & 0 \\ 0 & X'_1 \end{smallmatrix} \right)))^{\frac{k}{2}}
     e(i \M(\mbox{Im}(\tau +  \left(\begin{smallmatrix} 0 & 0 \\ 0 & X'_1 \end{smallmatrix} \right))^{-1})[
     \mbox{Im}(z + \left( \begin{smallmatrix} 0 \\  X'_3 \end{smallmatrix} \right)  )])     
     \right.
     \\
     &&
     \left.
    \times  e(\M(\omega+ X'_2))
    e(-\N \left(\begin{smallmatrix} X'_1 & X'_3 \\ ^t X'_3 & X'_2 \end{smallmatrix} \right) )
    \, dX'_1 \, dX'_3 \, dX'_2
    \ e(i \N v'^{-1}[y']) (\det v')^{\frac{k}{2}} 
    \right. \\
    &&
    \left.
    \times
    \det(v)^{-\frac{k}{2}} e(-i\M v^{-1}[y])
    \right| \\
    &< &
    C_\phi\, e(-i \N T')\, e(i \M T)
        \ e(i \N v'^{-1}[y']) \det(v')^{\frac{k}{2}} 
    \det(v)^{-\frac{k}{2}} e(-i \M v^{-1}[y]) \\
    &=&
    C_\phi\, e(-i \N(T' -  v'^{-1}[y'])) \, e(i \M(T - v^{-1}[y]))     
    \det(v')^{\frac{k}{2}} \det(v)^{-\frac{k}{2}} .
\end{eqnarray*}
We now write
$
\displaystyle{
  \N := \begin{pmatrix} N_2 & \frac12 R_2 \\ \frac12 {^t R_2} & \M \end{pmatrix}
}
$ and put
\begin{eqnarray*}
  \Delta := v_2 - v'^{-1}[y'_1],
  \quad
  \eta := y_2 - {^t y'_1} v'^{-1} y'_2.
\end{eqnarray*}
Then, by a straightforward calculation we have
\begin{eqnarray*}
 \mbox{Tr}(\N(T' - v'^{-1}[y'])) - \mbox{Tr}(\M(T-v^{-1}[y])) 
 &=&
 \mbox{Tr}\left( \N \begin{pmatrix} \Delta & \eta \\ {^t \eta} & \Delta^{-1}[\eta] \end{pmatrix} \right).
\end{eqnarray*}
Since
\begin{eqnarray*}
  \N &=&
  \begin{pmatrix} 
    N_2 - \frac14 \M^{-1}[ ^t R_2] & 0 \\ 0 & \M
  \end{pmatrix}
  \left[
    \begin{pmatrix}
      1_t & 0 \\
      \frac12 \M^{-1} {^t R_2} & 1_r
    \end{pmatrix}
  \right]
\end{eqnarray*}
and since
\begin{eqnarray*}
  \begin{pmatrix}
    \Delta & \eta \\ ^t \eta & \Delta^{-1}[\eta]
  \end{pmatrix}
  &=&
  \begin{pmatrix} 
    \Delta & 0 \\ 0 & 0
  \end{pmatrix}
  \left[
    \begin{pmatrix}
      1_t & \Delta^{-1} \eta \\
      0 & 1_r
    \end{pmatrix}
  \right]  ,
\end{eqnarray*}
there exists $\begin{pmatrix} v_2 & y_2 \end{pmatrix} \in \mbox{Sym}_t^+ \times \R^{(t,r)}$
which satisfies
$
\displaystyle{
\N \begin{pmatrix} \Delta & \eta \\ {^t \eta} & \Delta^{-1}[\eta] \end{pmatrix} 
=
\begin{pmatrix}
  1_t & 0 \\ 0 & 0
\end{pmatrix}
}
$.
By a straightforward calculation, such $\begin{pmatrix} v_2 & y_2 \end{pmatrix}$ is given by
\begin{eqnarray*}
v_2 &=& v'^{-1}[y'_1] + (N_2 - \frac14 \M^{-1}[^t R_2])^{-1}, \\
y_2 &=& - \frac12 (N_2 - \frac14 \M^{-1}[^t R_2])^{-1} \M^{-1} {^t R_2}
+ {^t y'_1} v'^{-1} y'_2.
\end{eqnarray*}
With this $\begin{pmatrix} v_2 & y_2 \end{pmatrix}$ we have
\begin{eqnarray*}
 &&
  |\phi_\N(\tau',z') e(i \N v'^{-1}[y']) (\det v')^{\frac{k}{2}} | \\
 &<&
    C_\phi\, e(-i t)
     \det(v')^{\frac{k}{2}} \det(v)^{-\frac{k}{2}} \\
  &=&
    C_\phi\, e(-i t)
     \det(v_2 - v'^{-1}[y'_1])^{-\frac{k}{2}}  \\
  &=&
    C_\phi\, e(-i t)
     \det(N_2 - \frac14 \M^{-1}[^t R_2])^{\frac{k}{2}} \\
  &=&
    C_\phi\, e(-i t)
     \det(\M)^{-\frac{k}{2}} \det(\N)^{\frac{k}{2}}   .
\end{eqnarray*}
Hence we conclude this lemma.
\end{proof}

Let $\phi_{\M}$ and $\psi_{\M}$ be Jacobi cusp forms of weight $k$ of index $\M$
on $\Gamma_n$.
Let $D_t(\phi_{\M}, \psi_{\M} ; s)$ be the Dirichlet series defined
in $(\ref{def:dirichlet})$.
\begin{lemma}\label{lem:Diri_conv}
The Dirichlet series $D_t(\phi_{\M}, \psi_{\M} ; s)$
converges absolutely with sufficient large $Re(s)$.
\end{lemma}
\begin{proof}
By virtue of Lemma~\ref{lem:FJc_norm} we have
\begin{eqnarray*}
  \langle \phi_\N, \psi_\N \rangle
  < C \det(\N)^{k}
\end{eqnarray*}
with a certain positive number $C$ which does not depend on the choice of $\N$.
Hence, it is enough to show that the series
\begin{eqnarray*}
   \sum_{\mathcal{N} 
    \in B_{t,r}(\Z) \backslash L_{t,r}^+(\M) 
    } \frac{1}{
    |\epsilon_{t,r}(\mathcal{N})| \det(\mathcal{N})^{s}}.
\end{eqnarray*}
converges absolutely for sufficiently large $Re(s)$.

It is known that the series
\begin{eqnarray*}
  \sum_{\T \in GL(t,\Z) \backslash L_{t}^+} \frac{1}{|\epsilon(\T)| \det(\T)^s}
\end{eqnarray*}
is absolutely convergent for $Re(s) > \frac{t+1}{2}$ (cf. \cite{Shintani}).

There exists a natural map
$B_{t,r}(\Z) \backslash L_{t,r}^+(\M)  \rightarrow GL(t,\Z) \backslash L_{t}^+$
given by
\begin{eqnarray*}
\N = 
\begin{pmatrix}
  N_2 & \frac12 R_2 \\
  ^t \frac12 R_2 & \M
\end{pmatrix}
 \rightarrow 4 l N_2 - l \M^{-1}[^t R_2],
\end{eqnarray*}
where $2l$ is the smallest positive integer which satisfies
$2l (2\M)^{-1}\Z^{(r,1)} \subset \Z^{(r,1)}$.
This map may not be injective, but if two matrices
$\N_i = \begin{pmatrix} N_{2,i} & \frac11 R_{2,i} \\ \frac12 ^t R_{2,i} & \M \end{pmatrix} \in L_{t,r}^+(\M)$ $(i=1,2)$ satisfy the conditions
\begin{eqnarray*}
  4 N_{2,1} - \M^{-1}[^t R_{2,1}]
  &=&
    4 N_{2,2} - \M^{-1}[^t R_{2,2}]
\end{eqnarray*}
and $R_{2,1} - R_{2,2} \in \Z^{(t,r)} (2\M)$, then $\N_1$ and $\N_2$ belong to the same
equivalent class in $B_{t,r}(\Z) \backslash L_{t,r}^+(\M) $.
Therefore, for a fixed representative $\T$ in $GL(t,\Z) \backslash L_{t}^+$,
there exist at most $\det(2\M)^t$ representatives in $B_{t,r}(\Z) \backslash L_{t,r}^+(\M) $
which map to $\T$.

For $\N = 
\begin{pmatrix}
  N_2 & \frac12 R_2 \\
  ^t \frac12 R_2 & \M
\end{pmatrix}
$, we remark the identity
\begin{eqnarray*}
  \det(\N)
  &=&
  (4l)^{-t} \det(\M) \det(4l N_2 - l \M^{-1}[^t R_2]).
\end{eqnarray*}
For any real number $s$, we have
\begin{eqnarray*}
\frac{1}{
    |\epsilon_{t,r}(\mathcal{N})| \det(\mathcal{N})^{s}}
  &\leq&
  \frac{1}{\det(\N)^s}   
  \, = \, 
  (4 l)^{ts} \det(\M)^{-s}
  \frac{|\epsilon(\N')|}{|\epsilon(\N')| \det(\N')^s},
\end{eqnarray*}
where we put $\N' = 4lN_2 - l \M^{-1}[^t R_2]$.

It is not difficult to show that there exists a constant $c$ such that
$|\epsilon(\N')| < \det(\N')^c$ for any $\N' \in L_t^+$.
Therefore, for sufficiently large real number $s$, we have
\begin{eqnarray*}
   \sum_{\mathcal{N} 
    \in B_{t,r}(\Z) \backslash L_{t,r}^+(\M) 
    } \frac{1}{
    |\epsilon_{t,r}(\mathcal{N})| \det(\mathcal{N})^{s}}
 < 
  (4l)^{ts}
  2^{rt}
   \det(\M)^{-s+t}
    \sum_{\N' \in GL(t,\Z) \backslash L_{t}^+} \frac{1}{|\epsilon(\N')| \det(\N')^{s-c}}.
\end{eqnarray*}
Thus we conclude this lemma.
\end{proof}

For $0 \leq t \leq n$, we put
\begin{eqnarray*}
  P_{n-t,t} := \left\{  \begin{pmatrix} * & * \\ 0^{(t,2n-t)} & * \end{pmatrix} \in \Gamma_n \right\}
  = \left\{ \left(\begin{smallmatrix} * & 0^{(n-t, t)}  & * & * \\
                                                    * & * & * & * \\
                                                    * & 0^{(n-t,t)} & * & * \\
                                                    0^{(t,n-t)}& 0^{(t,t)}  & 0^{(t,n-t)} & * \end{smallmatrix}\right) \in \Gamma_n \right\}.
\end{eqnarray*}
For $\tau \in \H_n$ and for $s \in \C$, we set
\begin{eqnarray*}
   E_t^{(n)}(s;\tau)
 &:=&
   \sum_{\gamma \in P_{n-t,t} \backslash \Gamma_n} \left(\frac{\det(\mbox{Im}(\gamma \cdot \tau))}{\det(\mbox{Im}((\gamma \cdot \tau)_1))}\right)^s,
\end{eqnarray*}
where $(\gamma \cdot \tau)_1$ is the left upper part of $\gamma \cdot \tau$ of size $(n-t) \times (n-t)$.
The series
$E_t^{(n)}(s;\tau)$
converges absolutely for $\mbox{Re}(s) > n - \frac{t-1}{2}$ (see \cite[p.42-43]{Ya3}).

For $R \in \Z^{(n,r)}$ we put
\begin{eqnarray}\label{id:vartheta_2}
  \vartheta_{\M,R}(\tau,z)
  &:=&
  \sum_{\begin{smallmatrix} p \in \Z^{(n,r)} \\ p \equiv R \!\! \mod \Z^{(n,r)}(2\M) \end{smallmatrix}}
  e\left(\frac14 p \M^{-1} {^t p} \tau + p {^t z} \right).
\end{eqnarray}
We remark that $\vartheta_{\M,R}$ is defined for $R$ modulo $\Z^{(n,r)} (2\M)$.

We take the following decompositions with theta series:
\begin{equation}\label{id:f_R_g_R}
\begin{split}
  \phi_{\M}(\tau,z)
  &=
  \sum_{R \!\! \mod \Z^{(n,r)} (2 \M) }
  f_R(\tau)  \, \vartheta_{\M,R}(\tau,z), \\
  \psi_{\M}(\tau,z)
  &=
  \sum_{R \!\! \mod \Z^{(n,r)} (2 \M) }
  g_R(\tau)  \, \vartheta_{\M,R}(\tau,z).
\end{split}
\end{equation}
We call it theta decomposition.
We also take the Fourier expansions of $\phi_\M$ and $\psi_\M$:
\begin{eqnarray*}
 \phi_\M(\tau,z) &=& \sum_{N,R} C_\phi(N,R) e(N \tau + R {^t z}), \\
 \psi_\M(\tau,z) &=& \sum_{N,R} C_\psi(N,R) e(N \tau + R {^t z}),
\end{eqnarray*}
where in the above summations $N \in L_n^+$ and $R \in \Z^{(n,r)}$
run over matrices which satisfy $4 N - \M^{-1}[^t R] > 0$.
Then we have
\begin{eqnarray*}
  f_R(\tau) &=& \sum_{N} C_\phi(N,R)\, e\!\left(\frac{1}{4} (4 N - \M^{-1}[^t R]) \tau\right), \\
  g_R(\tau) &=& \sum_{N} C_\psi(N,R)\, e\!\left(\frac{1}{4} (4 N - \M^{-1}[^t R]) \tau\right),
\end{eqnarray*}
where in the above summations $N \in L_n^+$ runs over matrices which satisfy the condition
$4N - \M^{-1}[^t R] > 0$.

\begin{prop}\label{prop:int_D}
  We have an integral expression of $D_t(\phi_{\M},\psi_{\M}; s)$ as follows.
  If $Re(s)$ is sufficiently large, then we obtain
 \begin{eqnarray*}
   &&
   (1+\delta_{t,n})^{-1} \pi^{-\frac14 t(t-1) + t(s + k - n + \frac{t - r - 1}{2})}
   \prod_{j=1}^t 
   \Gamma\!\left( s+k -n + \frac{t - r - j}{2} \right)^{-1}
   \\
   &&
  \times \int_{\Gamma_{n} \backslash \H_{n}}
    \sum_{R \!\! \mod \Z^{(n,r)} (2 \M)} f_R(\tau) \overline{g_R(\tau)} (\det(\mbox{Im}(\tau)))^{k-\frac{r}{2}}
    E_t^{(n)}(s;\tau) \, d\tau \\
  &=&
    \det(2\M)^{\frac{n-t}{2} + s+k-n+\frac{t-r-1}{2}}
    2^{\frac{(n-t)r}{2} - (2t+r)(s+k - n +\frac{t - r - 1}{2})}
  \\
  &&
    \times
    D_t(\phi_{\M},\psi_{\M}; s + k - n  + (t-r-1)/2).
 \end{eqnarray*}  
\end{prop}
We will show this proposition in \S\ref{s:proof_of_prop_int_D}.

We put
 $\displaystyle{\xi(s) := \pi^{-\frac{s}{2}} \Gamma\!\left(\frac{s}{2}\right) \zeta(s)  }$
and set
\begin{eqnarray*}
\mathcal{E}_t^{(n)}(s ; \tau)
&:=&
\prod_{i=1}^t \xi(2s+1-i) \prod_{i=1}^{[t/2]}
\xi(4s-2n+2t-2i) E_t^{(n)}(s ; \tau) .
\end{eqnarray*}

The following theorem has been shown by Kalinin~\cite{kalinin} for $t = n$
and by Yamazaki~\cite{Ya3} for $1 \leq t < n$.
\begin{theorem}[\cite{kalinin}, \cite{Ya3}]\label{th:eisenstein_series}
The function $\mathcal{E}_t^{(n)}(s ; \tau)$ has a meromorphic continuation
to the whole complex plane as the function of $s$ and holomorphic for $Re(s) > (2n-t+1)/2$.
Moreover, $\mathcal{E}_t^{(n)}(s ; \tau)$ satisfies the functional equation
\begin{eqnarray*}
  \mathcal{E}_t^{(n)}(s ; \tau)
  &=&
  \mathcal{E}_t^{(n)}\left(\frac{2n-t+1}{2} - s ; \tau\right).
\end{eqnarray*}
It has a simple pole at $s = n - (t-1)/2$ with the residue
\begin{eqnarray*}
   \frac{1+\delta_{t,n}}{2} \prod_{j=2}^t \xi(j) \prod_{j=1}^{[t/2]} \xi(2n-2t+2j+1)
\end{eqnarray*}
when $n > 1$ and with the residue $1/2$ when $n=t=1$.
(cf.~\cite[Theorem 2]{kalinin} for $t = n > 1$, ~\cite[Theorem 2.2]{Ya3}
for $1 \leq t < n$).

Moreover, 
if $t = n$, then the function
$\displaystyle{ \xi(2s) \prod_{i=1}^{[n/2]} \xi(4s-2i) E_n^{(n)}(s ; \tau)}$
has a meromorphic
 continuation to the
whole complex plane in $s$ except the possible poles of finite order at $s = j/4$
for integers $j$ $(0 \leq j \leq 2n+2)$.

If $t = 1$, then $\mathcal{E}_1^{(n)}(s ; \tau)$ has a meromorphic continuation to the
whole complex plane in $s$ except the poles at $s = n$ and $0$ with
residues $\frac12$ and $-\frac12$, respectively.
\end{theorem}
It is remarked in~\cite{Ya3} that if $t  \geq 2n - 2t + 2$, then
we can simplify the gamma factor of $\mathcal{E}_t^{(n)}$
by virtue of the cancellation of the above functional equation.
It means that it is possible to take
$\displaystyle{\prod_{i=1}^{2n-2t+1} \xi(2s+1-i) \prod_{i=1}^{[t/2]}
\xi(4s-2n+2t-2i) E_t^{(n)}(s ; \tau)}$
as the choice of the definition of $\mathcal{E}_t^{(n)}$ in this case.
The residue of $\mathcal{E}_t^{(n)}$ in the theorem will be changed
if we change the gamma factor.

We put
\begin{eqnarray*}
 &&
 \mathcal{D}_{t}(\phi_\M, \psi_\M ; s) \\
 &:=&
 (4\pi)^{-ts} (\det \M)^s \prod_{j=1}^t \left(\Gamma\left(s - \frac{j-1}{2}\right) \xi(2s-2k+2n+r+2-t-j)\right) \\
 &&
 \times \prod_{j=1}^{[t/2]} \xi(4s-4k+2n+2r+2-2j) \\
 &&
 \times D_{t}(\phi_\M, \psi_\M ; s) .
\end{eqnarray*}
We remark that if $r = 0$, then we regard $\det(\M)$ as $1$.

By virtue of Proposition~\ref{prop:int_D}
the function $\mathcal{D}_{t}(\phi_\M, \psi_\M ; s) $ equals to
\begin{eqnarray*}
&&
\int_{\Gamma_{n} \backslash \H_{n}}
    \sum_{R \!\! \mod \Z^{(n,r)} (2 \M)} f_R(\tau) \overline{g_R(\tau)} (\det(\mbox{Im} \tau))^{k-\frac{r}{2}}
    \mathcal{E}_t^{(n)}(s-k+n-(t-r-1)/2 ; \tau) \, d\tau
\end{eqnarray*}
times the constant $\pi^{-\frac14t(t-1)} \det(4\M)^{-\frac{n-t}{2}} (1 + \delta_{t,n})^{-1}$.

Thus, due to Theorem~\ref{th:eisenstein_series} we have the following.

\begin{theorem}\label{th:D_t_M}
The function $\mathcal{D}_t(\phi_{\M}, \psi_{\M} ; s)$
has a meromorphic continuation to the whole complex plane
and holomorphic for $Re(s) > k - \dfrac{r}{2}$.
It has a simple pole at $s = k - \dfrac{r}{2}$ with the residue
\begin{eqnarray*}
  (1+ \delta_{0,r})^{-1} \pi^{-\frac14 t (t-1)} \det(4\M)^{\frac{t}{2}} \langle \phi_\M, \psi_\M \rangle
  \prod_{j=2}^t \xi(j) \prod_{j=1}^{[t/2]} \xi(2n-2t+2j+1)
\end{eqnarray*}
when $n > 1$ and with the residue $\frac12 (1+\delta_{0,r})^{-1}  \det(4\M)^{\frac12} \langle \phi_\M, \psi_\M \rangle$
when $n = t = 1$.

It satisfies a functional equation
\begin{eqnarray*}
  \mathcal{D}_t(\phi_\M, \psi_\M ; s)
  &=&
  \mathcal{D}_t\!\left(\phi_\M, \psi_\M ; 2k-n-r+\frac{t-1}{2}-s\right).
\end{eqnarray*}

  Moreover, if $t = 1$, then
  $\mathcal{D}_1(\phi_\M,\psi_\M ; s)$ has a meromorphic continuation to
  the whole complex plane and holomorphic
  except for simple poles at $s = k -\dfrac{r}{2}$ and $s = k - \dfrac{r}{2} - n$.
  The residue at $s = k-\dfrac{r}{2}$ is 
  \begin{eqnarray*}
   (1 + \delta_{1,n})^{-1} (1 + \delta_{0,r})^{-1} \det(4 \M)^{\frac12}  \langle \phi_\M, \psi_\M \rangle.
  \end{eqnarray*}
The case $r = 0$ has been shown in~\cite{Ya3}.
\end{theorem}

We remark that if $n = r = t = 1$, then the above residue coincides with Proposition~1~(a) 
in~\cite{ImaMa}.
However, Proposition 1 (a) in~\cite{ImaMa}
 should read
\begin{eqnarray*}
\mbox{Res}_{s_2 = k - 1/2} D_{F,G}(s_1,s_2) = \pi^{k+\frac12}\, \Gamma\!\left(k-\frac12\right)^{-1}
\zeta(2)^{-1} L(F,G,s_1+k-1).
\end{eqnarray*}

After we shall explain some similar results of Theorem~\ref{th:D_t_M} for Jacobi cusp forms of \textit{half-integral weight}
in Section~\ref{s:plus_space_jacobi_forms},
we will prove Proposition~\ref{prop:int_D} in Section~\ref{s:proof_of_prop_int_D}.

%
\section{Rankin-Selberg method for the plus space of Jacobi forms}\label{s:plus_space_jacobi_forms}

In this section we shall explain the half-integral weight case.
In this section we assume that $k$ is an even integer.
We assume $r \geq 1$.
Let $\M = \begin{pmatrix} \M_1 & \frac12 L \\ \frac12 ^t L & 1  \end{pmatrix} \in L_r^+$
with $\M_1 \in L_{r-1}^+$ and $L \in M_{r-1,1}(\Z)$.
If $r \geq 2$, we set
\begin{eqnarray*}
  \mathfrak{M} := 4\M_1 - L {^t L}.
\end{eqnarray*}
If $r = 1$, then $\M = 1$ and
we regard $\mathfrak{M} = \emptyset $ as the empty set and we put $\det(\mathfrak{M}) = 1$
by abuse of notation.

We set $\Gamma_0^{(n)}(4) := \left. \left\{ \begin{pmatrix} A & B \\ C & D \end{pmatrix} \in \Gamma_n
\, \right| \, C \in 4 \Z^{(n,n)} \right\}$ .

Let $J_{k-\frac12,\mathfrak{M}}^{(n)+}$ be the plus-space of Jacobi forms
of weight $k-\frac12$ of index $\mathfrak{M}$ on $\Gamma_0^{(n)}(4)$
which is a generalization of generalized plus-space of Siegel modular forms
of weight $k-\frac12$ to Jacobi forms. The space $J_{k-\frac12,\mathfrak{M}}^{(n)+}$
is defined as follows.
Let $\phi$ be a Jacobi form of weight $k-\frac12$ of index $\mathfrak{M}$ on $\Gamma_0^{(n)}(4)$.
The reader is referred to~\cite{matrix_integer} for the precise definition
of Jacobi forms of half-integral weight.
We take the Fourier expansion
\begin{eqnarray*}
 \phi(\tau,z) &=&  \sum_{N',R'} C_\phi(N',R') e(N' \tau + R' {^t z})
\end{eqnarray*}
for $(\tau,z) \in \H_n \times \C^{(n,r-1)}$,
where $N'$ and $R'$ run over $L_n^*$ and $\Z^{(n,r-1)}$, respectively, such that
$4 N' - R' \mathfrak{M}^{-1} {^t R'} \geq 0$.
Then $\phi$ belongs to $J_{k-\frac12,\mathfrak{M}}^{(n)+}$ if and only if 
$C_\phi(N',R') = 0$ unless
\begin{eqnarray*}
  \begin{pmatrix}
    N' & \frac12 R' \\ \frac12 {^t R'} & \mathfrak{M}
  \end{pmatrix}
  \equiv \lambda {^t \lambda} \mod 4
\end{eqnarray*}
with some $\lambda \in \Z^{(n+r-1,1)}$.

If $r = 1$, then the space $J_{k-\frac12,\mathfrak{M}}^{(n)+}$ coincides with
the generalized plus-space of Siegel modular forms.
There exists a  linear isomorphism map $\iota_\M$ from $J_{k,\M}^{(n)}$
to $J_{k-\frac12,\mathfrak{M}}^{(n)+}$ (cf.~\cite{EZ} (for $r=n=1$), ~\cite{Ib} (for $r =1$, $n > 1$), ~\cite{matrix_integer} (for $r > 1$, $n \geq 1$)).
This map $\iota_\M$ is given as follows.

Let $\phi_\M \in J_{k,\M}^{(n)}$ be a Jacobi form.
We denote by $C_{\phi_\M}( * , * )$ the Fourier coefficients of $\phi_\M$.
For $\tau \in \H_n$ and for $z = (z_1,z_2) \in \C^{(n,r)}$ ($z_1 \in \C^{(n,r-1)}$, $z_2 \in \C^{(n,1)}$),
we take the theta decomposition
\begin{eqnarray*}
  \phi_\M(\tau,z)
  &=&
   \sum_{\begin{smallmatrix} R \in \Z^{(n,1)} \\ R \!\! \mod 2 \Z^{(n,1)} \end{smallmatrix}} f_{R,\M_1}(\tau,z_1) \vartheta_{1,L,R}(\tau,z_1,z_2),
\end{eqnarray*}
where
\begin{eqnarray*}
  f_{R,\M_1}(\tau,z_1)
  &=&
  \sum_{N_1 \in L_{n}^*, N_3 \in \Z^{(n,r-1)}}
  C_{\phi_\M}(N_1, \begin{pmatrix}  N_3 &  R  \end{pmatrix})
  \\
  && \times
  e( (N_1 - \frac14 R {^t R})\tau + (N_3 - \frac12 R {^t L}) {^t z_1})
\end{eqnarray*}
and the function $\vartheta_{1,L,R}$ will be denoted in (\ref{id:vartheta_3}) (cf. \cite[Lemma 4.1]{matrix_integer}).
We put
\begin{eqnarray*}
  \iota_\M(\phi_\M)(\tau,z_1)
  & = & 
  \sum_{R \in \Z^{(n,1)}/(2 \Z^{(n,1)})} f_{R,\M_1}(4\tau, 4 z_1).
\end{eqnarray*}
For the sake of simplicity we write  $\phi_\mathfrak{M} = \iota_\M(\phi_\M)$.
Then $\phi_\mathfrak{M}$ belongs to $J_{k-\frac12,\mathfrak{M}}^{(n)+}$
(cf.~\cite[Proposition 4.4]{matrix_integer}).
If $\phi_\M$ is a Jacobi cusp form, then $\phi_\mathfrak{M}$ is also a Jacobi cusp form.
If $r=1$, then $\phi_\mathfrak{M}$ is a Siegel modular form
(cf. \cite{EZ}, \cite{Ib}).

Let $\phi$ and $\psi$ be Jacobi cusp forms of weight $k-\frac12$ of index $S \in L_r^+$
on $\Gamma_0^{(n)}(4)$.
The Petersson inner product is defined by
\begin{eqnarray*}
 \langle \phi, \psi \rangle 
 &:=&
 \left[ \Gamma_n : \Gamma_0^{(n)}(4) \right]^{-1}
   \int_{\mathcal{F}_{n,r,4}}
    \phi(\tau,z) \overline{ \psi(\tau,z)} e^{-4\pi Tr( S v^{-1}[y])} \det(v)^{k - n - r - \frac32 } 
    \, du\, dv\, dx\, dy,
\end{eqnarray*}
where $\mathcal{F}_{n,r,4} :=
\Gamma_{n,r}^J(4) \backslash (\H_n \times \C^{(n,r)}) $,
$\tau = u + i v$, $z = x + i y$,
$du = \prod_{i \leq j} u_{i,j}$,  $d v = \prod_{i \leq j} v_{i,j}$,
$dx = \prod_{i,j} x_{i,j}$, $dy = \prod_{i,j} y_{i,j}$
and $\left[ \Gamma_n : \Gamma_0^{(n)}(4) \right]$ denotes the index of
$\Gamma_0^{(n)}(4)$ in $\Gamma_n$.
Here we put
\begin{eqnarray*}
  \Gamma_{n,r}^J(4)
&:=&
  \left. \left\{
  \begin{pmatrix}
    A & 0 & B  & * \\
    * & 1_r & * & * \\
    C & 0 & D & * \\
    0 & 0 & 0 & 1_r
  \end{pmatrix}
  \in \Gamma_{n+r}
  \, \right| \,
  \begin{pmatrix}
    A & B \\ C & D 
  \end{pmatrix}
  \in \Gamma_0^{(n)}(4)
  \right\} .
\end{eqnarray*}

\begin{lemma}\label{lem:matrix_jacobi_petersson}
Let $\phi_\M$ and $\psi_\M$ be Jacobi cusp forms in $J_{k,\M}^{(n)\, cusp}$.
We put $\phi_\mathfrak{M} = \iota_\M(\phi_\M)$ and $\psi_\mathfrak{M} = \iota_\M(\psi_\M)$.
As for the Petersson inner product we obtain the identity
 \begin{eqnarray*}
  \langle \phi_\M, \psi_\M \rangle
  &=&
  (1 + \delta_{1,r})^{-1} 2^{2n(k-1)}  \langle \phi_\mathfrak{M}, \psi_\mathfrak{M} \rangle.
 \end{eqnarray*}
We remark that, in the case of $r = 1$ and $n = 1$, this identity has been obtained
by combining Kohnen and Zagier~\cite[p.p. 189--191]{KoZa} and Eichler-Zagier~\cite[Theorem 5.3]{EZ}. Remark that the denominator of RHS of~\cite[Theorem 5.3]{EZ} is neither
$\sqrt{2m}$ nor $\sqrt{4m}$ but $4\sqrt{m}$. 
We remark that, in the case of $r =1 $and $n > 1$, the above identity has been obtained
by Katsurada and Kawamura~\cite[p. 2051 (6)]{KaKa:Dirichlet}.
Remark that $2^{(2k-2)(n-1)}$ in ~\cite[p. 2051 (6)]{KaKa:Dirichlet} should read
$2^{(2k-2)(n-1)-1}$.
\end{lemma}
\begin{proof}
We recall the symbol $P_{0,n} = \left\{ \left( \begin{smallmatrix} * & * \\ 0^{(n,n)} & * \end{smallmatrix} \right) 
\in \Gamma_n \right\}$.
We set
\begin{eqnarray*}
  E_{n,4}^{(n)}(s ; \tau)
  &:=&
  \sum_{\gamma \in P_{0,n} \backslash \Gamma_0^{(n)}(4)}
  \det(\mbox{Im}(\gamma \cdot \tau))^s.
\end{eqnarray*}
We take the theta decompositions
\begin{eqnarray*}
  \phi_\mathfrak{M}(\tau,z_1)
  &=&
  \sum_{R \mod \Z^{(n,r-1)}(2\mathfrak{M})} \tilde{f}_R(\tau) \vartheta_{\mathfrak{M},R}(\tau,z_1), \\
    \psi_\mathfrak{M}(\tau,z_1)
  &=&
  \sum_{R \mod \Z^{(n,r-1)}(2\mathfrak{M})} \tilde{g}_R(\tau) \vartheta_{\mathfrak{M},R}(\tau,z_1).
\end{eqnarray*}
We put
\begin{eqnarray*}
  I_n(\phi_\mathfrak{M}, \psi_\mathfrak{M} ; s) 
  &:=&
  \int_{\Gamma_0^{(n)}(4) \backslash \H_n}
  \sum_{R \mod \Z^{(n,r-1)}(2\mathfrak{M})}
  \tilde{f}_R(\tau) \overline{\tilde{g}_R(\tau)} (\det(\mbox{Im}\tau))^{k-\frac12-\frac{r-1}{2}}
  E_{n,4}^{(n)}(s ; \tau) \, d\, \tau.
\end{eqnarray*}
where we put $\tau = u + \sqrt{-1} v$ and
$d\tau := \det(v)^{-n-1} d u \, d v$ and $d u = \prod_{l \leq m} u_{l,m}$, $dv = \prod_{l \leq m} v_{l,m}$. Here $u = (u_{l,m})$ and $v = (v_{l,m})$.
We put
\begin{eqnarray*}
 I_n(\phi_\M,\psi_\M;s)
 &:=&
  \int_{\Gamma_{n} \backslash \H_{n}}
    \sum_{R \!\! \mod \Z^{(n,r)} (2 \M)} f_R(\tau) \overline{g_R(\tau)} \det(\mbox{Im}(\tau))^{k-\frac{r}{2}}
    E_n^{(n)}(s ; \tau) \, d\tau,
\end{eqnarray*}
where $f_R$ and $g_R$ are denoted  in~(\ref{id:f_R_g_R}) through the decompositions of
$\phi_\M$ and $\psi_\M$ with the theta series.

We will show this lemma by comparing the residue of $  I_n(\phi_\mathfrak{M}, \psi_\mathfrak{M} ; s)
$ with the one of $  I_n(\phi_\M, \psi_\M ; s)$ at $s = \frac{n+1}{2}$.
Let
\begin{eqnarray*}
  \phi_\mathfrak{M}(\tau,z_1)
  &=&
  \sum_{\left( \begin{smallmatrix} N'_2 & \frac12 R'_2 \\ \frac12 ^t R'_2 & \mathfrak{M} \end{smallmatrix} \right) \in L_{n,r-1}^+(\mathfrak{M}) } C_{\phi_{\mathfrak{M}}}(N'_2,R'_2) e(N'_2 \tau + R'_2 {^t z_1})
\end{eqnarray*}
and
\begin{eqnarray*}
  \phi_\M(\tau,z)
  &=&
  \sum_{\left( \begin{smallmatrix} N_2 & \frac12 R_2 \\ \frac12 ^t R_2 & \M \end{smallmatrix} \right) \in L_{n,r-1}^+(\M) } A_{\phi_\M}(N_2,R_2) e(N_2 \tau + R_2 {^t z})
\end{eqnarray*}
be the Fourier expansions of $\phi_{\mathfrak{M}}$ and $\phi_\M$, respectively.
We similarly denote by $C_{\psi_\mathfrak{M}}(N'_2,R'_2)$ and $A_{\psi_\M}(N_2,R_2)$
the Fourier coefficients of $\psi_\mathfrak{M}$ and $\psi_\M$, respectively. 
We remark that if
\begin{eqnarray}\label{id:mat_int_1}
  \begin{pmatrix}
    N'_2 & \frac12 R'_2 \\ \frac12 ^t R'_2 & \mathfrak{M}
  \end{pmatrix}
  =
  4 \begin{pmatrix} N_2 & \frac12 R_{2,1} \\  \frac12 {^t R_{2,1}} & \M_1 \end{pmatrix}
  -
  \begin{pmatrix} R_{2,2} \\ L \end{pmatrix}
  ^t \begin{pmatrix} R_{2,2} \\ L \end{pmatrix}
\end{eqnarray}
and if
\begin{eqnarray}\label{id:mat_int_2}
  \begin{pmatrix}
    N_2 & \frac12 R_2  \\ 
    \frac12 {^t R_2} & \M 
  \end{pmatrix}  
  &=& 
  \begin{pmatrix}
    N_2 & \frac12 R_{2,1} & \frac12 R_{2,2} \\ 
    \frac12 {^t R_{2,1}} & \M_1 & \frac12 L \\
    \frac12 {^t R_{2,2}} & \frac12 ^t L & 1 
  \end{pmatrix},
\end{eqnarray}
then $C_{\phi_\mathfrak{M}}(N'_2,R'_2) = A_{\phi_\M}(N_2,R_2)$ and $C_{\psi_\mathfrak{M}}(N'_2,R'_2) = A_{\psi_\M}(N_2,R_2)$.

By the similar argument of the proof of the identity~(\ref{id:I_t}) in Proposition~\ref{prop:id:I_t}
which will be appeared in~\S\ref{s:proof_of_prop_int_D},
we have
\begin{eqnarray*}
  &&
  I_n(\phi_\mathfrak{M}, \psi_\mathfrak{M} ; s)  \\
  &=&
  2 \pi^{\frac14 n (n-1)} \prod_{i=1}^n \Gamma(s+k-\frac12-\frac{r-1}{2}-\frac{n+1}{2}-\frac{i-1}{2}) \\
  &&
  \times
      \sum_{ \mathfrak{N} \in L_{n,r-1}^+(\mathfrak{M}) \slash B_{n,r-1}(\Z)}
    \frac{1}{|\epsilon_{n,r-1}(\mathfrak{N})|} 
    C_{\phi_\mathfrak{M}}(N'_2,R'_2) \overline{C_{\psi_\mathfrak{M}}(N'_2,R'_2)}  \\
    && \times
      \det(\pi(4N'_2 - \mathfrak{M}^{-1}[^t R'_2]))^{-k+\frac12+\frac{r-1}{2}-s+\frac{n+1}{2}},
\end{eqnarray*}
where in the summation
we set $\mathfrak{N} = \begin{pmatrix} N'_2 & \frac12 R'_2 \\ \frac12 {^t R'_2} & \mathfrak{M}
\end{pmatrix}$.
We remark that if the identities (\ref{id:mat_int_1}) and (\ref{id:mat_int_2}) hold, then
\begin{eqnarray*}
  \det(\pi(4N'_2 - \mathfrak{M}^{-1}[^t R'_2]))
  &=&
  2^{2n} \det(\pi(4N_2 - \M^{-1}[^t R_2]))
\end{eqnarray*}
and $|\epsilon_{n,r-1}(\mathfrak{N})| = |\epsilon_{n,r}(\N)|$,
where we put
$\N = \begin{pmatrix} N_2 & \frac12 R_2 \\ \frac12 ^t R_2 & \M \end{pmatrix}$.
The map
$\begin{pmatrix} N_2 & \frac12 R_2 \\ \frac12 ^t R_2 & \M \end{pmatrix}
  \in L_{n,r}^+(\M)\slash B_{n,r}(\Z) \mapsto
  \begin{pmatrix} N'_2 & \frac12 R'_2 \\ \frac12 ^t R'_2 & \mathfrak{M} \end{pmatrix}
  \in L_{n,r-1}^+(\mathfrak{M})\slash B_{n,r}(\Z)$
  given by the identities (\ref{id:mat_int_1}) and (\ref{id:mat_int_2}) is bijective.
Therefore, by using the identity (\ref{id:I_t}) which will be appeared in~\S\ref{s:proof_of_prop_int_D}, 
we have
\begin{eqnarray}
  \notag
  &&
  I_n(\phi_\mathfrak{M}, \psi_\mathfrak{M} ; s)  \\
  \notag
  &=&
  2 \pi^{\frac14 n (n-1)} \prod_{i=1}^n \Gamma(s+k-\frac12-\frac{r-1}{2}-\frac{n+1}{2}-\frac{i-1}{2}) \\
  \notag
  &&
  \times
      \sum_{ \N \in L_{n,r}^+(\M) \slash B_{n,r}(\Z)}
    \frac{1}{|\epsilon_{n,r}(\N)|} 
    A_{\phi_\M}(N_2,R_2) \overline{A_{\psi_\M}(N_2,R_2)}  \\
   \notag
    && \times
    2^{2n(-k+\frac{r}{2}-s+\frac{n+1}{2})}  \det(\pi(4N_2 - \M^{-1}[^t R_2]))^{-k+\frac12+\frac{r-1}{2}-s+\frac{n+1}{2}} \\
  \label{id:In_M}
  &=&
  2^{2n(-k+\frac{r}{2}-s+\frac{n+1}{2})}\, 
  I_n(\phi_\M, \psi_\M ; s) .
\end{eqnarray}
We put
\begin{eqnarray*}
  \gamma_n(s) &:=& \prod_{i=1}^{n} \xi(2s + 1 - i) \prod_{i=1}^{[n/2]} \xi(4s-2i),
\end{eqnarray*}
where $\displaystyle{ \xi(s) = \pi^{-\frac{s}{2}} \Gamma\!\left(\frac{s}{2}\right) \zeta(s) }$ 
is the symbol denoted before Theorem~\ref{th:eisenstein_series},
and we put
\begin{eqnarray*}
  \mathcal{E}_{n,4}^{(n)}(s ; \tau)
  &:=&
  \gamma_n(s) E_{n,4}^{(n)}(s ; \tau) .
\end{eqnarray*}
Then $\mathcal{E}_{n,4}^{(n)}(s ; \tau)$ has a meromorphic continuation to the whole complex plane
in $s$ (cf.~\cite[Theorem 1]{kalinin}).
Moreover, $\mathcal{E}_{n,4}^{(n)}(s ; \tau)$ has a simple pole at $s = (n+1)/2$ with the residue
\begin{eqnarray*}
 \mbox{Res}_{s = \frac{n+1}{2}} \mathcal{E}_{n,4}^{(n)}(s ; \tau)
 &=&
 \frac{1}{[\Gamma_n : \Gamma_0^{(n)}(4)]}
 \mbox{Res}_{s = \frac{n+1}{2}} \mathcal{E}_{n}^{(n)}(s ; \tau) \\
 &=&
 \frac{1}{[\Gamma_n : \Gamma_0^{(n)}(4)]}
 \prod_{j=2}^{n} \xi(j) \prod_{j=1}^{[n/2]} \xi(2j+1).
\end{eqnarray*}
Therefore the residue of $\gamma_n(s) I_n(\phi_\mathfrak{M}, \psi_\mathfrak{M} ; s) $
at $s = \frac{n+1}{2}$ is
\begin{eqnarray*}
 &&
 \frac{1}{[\Gamma_n : \Gamma_0^{(n)}(4)]}
 \prod_{j=2}^{n} \xi(j) \prod_{j=1}^{[n/2]} \xi(2j+1)  \\
 && \times
     \int_{\Gamma_0^{(n)}(4) \backslash \H_n}
  \sum_{R \mod \Z^{(n,r-1)}(2\mathfrak{M})}
  \tilde{f}_R(\tau) \overline{\tilde{g}_R(\tau)} (\det(\mbox{Im}\tau))^{k-\frac12-\frac{r-1}{2}} \, d\, \tau \\
  &=&
  \frac{2}{1 + \delta_{0,r-1}}
  \det(4\mathfrak{M})^{\frac{n}{2}}
 \prod_{j=2}^{n} \xi(j) \prod_{j=1}^{[n/2]} \xi(2j+1) \langle \phi_\mathfrak{M}, \psi_\mathfrak{M} \rangle.  
\end{eqnarray*}
On the other hand the residue of $\gamma_n(s) I_n(\phi_\M,\psi_\M)$ at $s = \frac{n+1}{2}$ is
\begin{eqnarray*}
    2 \det(4\M)^{\frac{n}{2}}
 \prod_{j=2}^{n} \xi(j) \prod_{j=1}^{[n/2]} \xi(2j+1) \langle \phi_\M, \psi_\M \rangle.  
\end{eqnarray*}
We remark the identity $\det(4\mathfrak{M}) = 2^{2r-4} \det(4 \M)$.
Thus, by virtue of the identity~(\ref{id:In_M}), we have the lemma.
\end{proof}

Let $\phi_\mathfrak{M}, \psi_\mathfrak{M} \in J_{k-\frac12,\mathfrak{M}}^{(n)\, cusp} $ 
be Jacobi cusp forms of weight $k-\frac12$ with the index $\mathfrak{M} \in L_{r-1}^+$
on $\Gamma_0^{(n)}(4)$.
We remark that if $r =1$, then $\phi_\mathfrak{M}$ and $\psi_\mathfrak{M}$ 
are Siegel cusp forms of weight $k-\frac12$.
For any natural number $t$ $(1 \leq t \leq n)$, we take the Fourier-Jacobi expansions
\begin{eqnarray*}
  \phi_\mathfrak{M}(\tau,z) e(\mathfrak{M} \omega) &=& 
    \sum_{\mathfrak{N} \in L_{t,r-1}^+(\mathfrak{M})} \phi_{\mathfrak{N}}(\tau',z') e(\mathfrak{N} \omega'), \\
  \psi_{\mathfrak{M}}(\tau,z) e(\mathfrak{M} \omega)
  &=&
  \sum_{\mathfrak{N} \in L_{t,r-1}^+(\mathfrak{M})} \psi_{\mathfrak{N}}(\tau',z') e(\mathfrak{N} \omega').
\end{eqnarray*}
For complex number $s$ which real part is sufficient large, we set
\begin{eqnarray*}
  D_t(\phi_\mathfrak{M}, \psi_\mathfrak{M} ; s)
  &:=&
  \sum_{\mathfrak{N} \in B_{t,r-1}(\Z) \backslash L_{t,r-1}^+(\mathfrak{M})}
   \frac{\langle  \phi_{\mathfrak{N}}, \psi_{\mathfrak{N}} \rangle}{
    |\epsilon_{t,r-1}(\mathfrak{N})| \det(\mathfrak{N})^{s}}.
\end{eqnarray*}

\begin{lemma}
We assume that $\phi_\mathfrak{M}$ and $\psi_\mathfrak{M}$ belong to the plus space
$J_{k-\frac12,\mathfrak{M}}^{(n)+\, cusp} $.
Let $\phi_\M$ and $\psi_\M$ $\in J_{k,\M}^{(n)\, cusp} $
be Jacobi cusp forms which satisfy $\phi_\mathfrak{M} = \iota_\M(\phi_\M)$ and $\psi_\mathfrak{M} = \iota_\M(\psi_\M)$.
Then, for any $t$ $(1 \leq t \leq n)$, we have
\begin{eqnarray*}
 D_t(\phi_\mathfrak{M},\psi_\mathfrak{M} ; s)  &=&
  (1 + \delta_{1,r}) 2^{-2(k-1)(n-t) - 2(r+t-1)s} D_t(\phi_\M, \psi_\M ; s).
\end{eqnarray*}
\end{lemma}
In the case of $r = 1$, $\M = 1$ and $t = n$ this lemma has been shown in~\cite{KaKa}.
\begin{proof}
 Assume $\N$ is a matrix in $L_{t,r}^+(\M)$.
 Let $\phi_\N$ and $\psi_\N$ be the $\N$-th Fourier-Jacobi coefficients of $\phi_\M$ and $\psi_\M$,
 respectively.
 We put $\phi_\mathfrak{N} = \iota_\N(\phi_\N)$ and $\psi_\mathfrak{N} = \iota_\N(\psi_\N)$.
 Then $\phi_\mathfrak{N}$ and $\phi_\mathfrak{N}$ are $\mathfrak{N}$-th Fourier-Jacobi
 coefficients of $\phi_\mathfrak{M}$ and $\psi_\mathfrak{M}$, respectively.
 We remark that $\phi_\mathfrak{N}$ and $\psi_\mathfrak{N}$ belong to $J_{k-\frac12,\mathfrak{N}}^{(n-t)+\, cusp} $
 and remark that $\phi_\N$ and $\psi_\N$ belong to $J_{k,\N}^{(n-t)+\, cusp} $.
  
 By virtue of Lemma~\ref{lem:matrix_jacobi_petersson},
 we have
 $\langle \phi_\mathfrak{N}, \psi_\mathfrak{N} \rangle =
  (1 + \delta_{1,r}) 2^{-2(n-t)(k-1)}  \langle \phi_\N, \psi_\N \rangle$.
 We have also
 $|\epsilon_{t,r-1}(\mathfrak{N})| = |\epsilon_{t,r}(\N)| $ and
 $\det\mathfrak{N} = 2^{2(r+t-1)} \det\N$.
 Thus we conclude the lemma.
\end{proof}

We set
\begin{eqnarray*}
 &&
 \mathcal{D}_{t}(\phi_\mathfrak{M}, \psi_\mathfrak{M} ; s) \\
 &:=&
 \pi^{-ts} (\det \mathfrak{M})^s \prod_{j=1}^t \left(\Gamma\left(s - \frac{j-1}{2}\right) \xi(2s-2k+2n+r+2-t-j)\right) \\
 &&
 \times \left( \prod_{j=1}^{[t/2]} \xi(4s-4k+2n+2r+2-2j) \right)
  D_{t}(\phi_\mathfrak{M}, \psi_\mathfrak{M} ; s).
\end{eqnarray*}
Then 
\begin{eqnarray*}
  \mathcal{D}_{t}(\phi_\mathfrak{M}, \psi_\mathfrak{M} ; s) 
  &=&
  (1 + \delta_{1,r}) 2^{-2(k-1)(n-t)} \mathcal{D}_{t}(\phi_\M, \psi_\M ; s) .
\end{eqnarray*}

Due to Theorem~\ref{th:D_t_M} we have the followings.
%
\begin{theorem}\label{th:D_t_M_half}
%
The function $\mathcal{D}_t(\phi_{\mathfrak{M}}, \psi_{\mathfrak{M}} ; s)$
has a meromorphic continuation to the whole complex plane
and holomorphic for $Re(s) > k - \dfrac{r}{2}$.
It has a simple pole at $s = k - \dfrac{r}{2}$ with the residue
\begin{eqnarray*}
  (1 + \delta_{1,r})^{-1} 2^{2tk-t} \pi^{-\frac14 t (t-1)} \det(\mathfrak{M})^{\frac{t}{2}} \langle \phi_\mathfrak{M}, \psi_\mathfrak{M} \rangle
  \prod_{j=2}^t \xi(j) \prod_{j=1}^{[t/2]} \xi(2n-2t+2j+1)
\end{eqnarray*}
when $n > 1$ and with the residue
$2^{2k-3}  \det(\mathfrak{M})^{\frac12} \langle \phi_\mathfrak{M}, \psi_\mathfrak{M} \rangle$
when $n = t = 1$.

It satisfies the functional equation
\begin{eqnarray*}
  \mathcal{D}_t(\phi_\mathfrak{M}, \psi_\mathfrak{M} ; s)
  &=&
  \mathcal{D}_t\!\left(\phi_\mathfrak{M}, \psi_\mathfrak{M} ; 2k-n-r+\frac{t-1}{2}-s\right).
\end{eqnarray*}
Moreover, if $t = 1$, then
  $\mathcal{D}_1(\phi_\mathfrak{M},\psi_\mathfrak{M} ; s)$ has a meromorphic continuation to
  the whole complex plane and holomorphic
  except for simple poles at $s = k -\dfrac{r}{2}$ and $k - \dfrac{r}{2} - n$.
  The residue at $s = k-\dfrac{r}{2}$ is 
  \begin{eqnarray*}
   (1 + \delta_{1,n})^{-1} (1 + \delta_{1,r})^{-1} 2^{2k-1} \det(\mathfrak{M})^{\frac12}
   \langle \phi_\mathfrak{M}, \psi_\mathfrak{M} \rangle.
  \end{eqnarray*}
\end{theorem}

In particular, if $r = 1$ and $\M = 1$, 
then the space $J_{k,1}^{(n)\, cusp}$ of Jacobi cusp forms of degree $n$
is linearly isomorphic to the generalized plus-space $S_{k-\frac12}^{+}(\Gamma_0^{(n)}(4))$
as Hecke algebra modules.
Here $S_{k-\frac12}^{+}(\Gamma_0^{(n)}(4))$ is a certain subspace
of Siegel cusp forms of weight $k-\frac12$ of degree $n$
(see \cite{Ib} for the definition and the isomorphism).
We have the following.

\begin{cor}\label{cor:dirichlet_plus_space}
Let $F$, $G$ $\in S_{k-\frac12}^{+}(\Gamma_0^{(n)}(4))$.
The function $\mathcal{D}_t(F, G ; s)$
has a meromorphic continuation to the whole complex plane
and holomorphic for $Re(s) > k - \dfrac{1}{2}$.
It has a simple pole at $s = k - \dfrac{1}{2}$ with the residue
\begin{eqnarray*}
  2^{2tk-t-1} \pi^{-\frac14 t (t-1)}  \langle F, G \rangle
  \prod_{j=2}^t \xi(j) \prod_{j=1}^{[t/2]} \xi(2n-2t+2j+1)
\end{eqnarray*}
when $n > 1$ and with the residue
$2^{2k-3}   \langle F, G \rangle$
when $n = t = 1$.

It satisfies the functional equation
\begin{eqnarray*}
  \mathcal{D}_t(F, G ; s)
  &=&
  \mathcal{D}_t\!\left(F, G ; 2k-n+\frac{t-3}{2}-s\right).
\end{eqnarray*}
Moreover, if $t = 1$, then
  $\mathcal{D}_1(F, G ; s)$ has a meromorphic continuation to
  the whole complex plane and holomorphic
  except for simple poles at $s = k -\dfrac{1}{2}$ and $k - \dfrac{1}{2} - n$.
  The residue at $s = k-\dfrac{1}{2}$ is 

\begin{eqnarray*}
 Res_{s=k-\frac12} \mathcal{D}_1(F,G;s)
 &=&
 Res_{s=k-\frac12}( D_1(F,G;s) \pi^{-s} \Gamma(s) \xi(2s-2k+2n+1) ) \\
 &=&
 (1 + \delta_{1,n})^{-1} 2^{2(k-1)} \langle F, G \rangle .
\end{eqnarray*}
\end{cor}
We remark that the case $t = n$ in Corollary~\ref{cor:dirichlet_plus_space}
has been shown in~\cite{KoZa} (for $n=1$) and in~\cite{KaKa} (for $n>1$).

\section{Proof of Proposition~\ref{prop:int_D}}\label{s:proof_of_prop_int_D}
In this section we shall prove Proposition~\ref{prop:int_D}.
We use the same notation in \S\ref{s:jacobi_int}.
For $\tau \in \H_n$, we decompose $\tau$ as
$  \tau = \begin{pmatrix} \tau_1 & z'_1 \\ ^t z'_1 & \tau_2 \end{pmatrix},
  \tau_1 \in \H_{n-t}, \tau_2 \in \H_{t}, z'_1 \in \C^{(n-t,t)}$.
We write $\tau = u + i v$, $\tau_j = u_j + i v_j$ $(j = 1,2)$ and $z'_1 = x'_1 + i y'_1$ with matrices
$u,v, u_j, v_j, x'_1, y'_1$ which entries are real numbers.
For $\tau_1 \in \H_{n-t}$, we fix a fundamental domain
\begin{eqnarray*}
  D_t(\tau_1)
  &:=& \C^{(n-t,t)} \slash (\tau_1 \Z^{(n-t,t)} + \Z^{(n-t,t)})
\end{eqnarray*}  
and put
\begin{eqnarray*}
  \widetilde{D_t(\tau_1)}
  &:=&  \left\{ (z'_1 , \tau_2) \in D_t(\tau_1) \times \H_t \, \left| \, 
  \begin{pmatrix} \tau_1 & z'_1 \\ ^t z'_1 & \tau_2 \end{pmatrix} \in \H_n \right\} \right. .
\end{eqnarray*}
The group $P_{0,t} = \left\{ \begin{pmatrix} A & B \\ 0^{(t,t)} & {^t A}^{-1} \end{pmatrix} \in \Gamma_t \right\}$ acts on $\widetilde{D_t(\tau_1)}$ by
\begin{eqnarray*}
  \begin{pmatrix} A & B \\ 0^{(t,t)} & {^t A}^{-1} \end{pmatrix} \cdot (z'_1, \tau_2) 
  := (z'_1 {^t A}, \tau_2[^t A] + B {^t A})
\end{eqnarray*}
for $\begin{pmatrix} A & B \\ 0^{(t,t)} & {^t A}^{-1} \end{pmatrix}  \in P_{0,t}$ and
for $(z'_1,\tau_2) \in \widetilde{D_t(\tau_1)}$.

We put
\begin{eqnarray*}
 I_t(\phi_\M,\psi_\M;s)
 &:=&
  \int_{\Gamma_{n} \backslash \H_{n}}
    \sum_{R \!\! \mod \Z^{(n,r)} (2 \M)} f_R(\tau) \overline{g_R(\tau)} \det(v)^{k-\frac{r}{2}}
    E_t^{(n)}(s ; \tau) \, d\tau,
\end{eqnarray*}
where $f_R$ and $g_R$ are denoted  in~(\ref{id:f_R_g_R}) through the decompositions of
$\phi_\M$ and $\psi_\M$ with the theta series,
and where we put $d\tau := \det(v)^{-n-1} d u \, d v$ and $d u = \prod_{l \leq m} u_{l,m}$, $dv = \prod_{l \leq m} v_{l,m}$. Here $u = (u_{l,m})$ and $v = (v_{l,m})$.

We have
\begin{eqnarray*}
  &&
    I_t(\phi_\M,\psi_\M;s)
 \\
  &=& 
  \int_{P_{n-t,t} \backslash \H_{n}}
    \sum_{R \!\! \mod \Z^{(n,r)} (2 \M)} f_R(\tau) \overline{g_R(\tau)} \det(v)^{k-\frac{r}{2}+s} \det(v_1)^{-s} \, d\tau    
   \\
  &=& 
  \frac{1 + \delta_{t,n}}{2}
  \int_{\Gamma_{n-t} \backslash \H_{n-t}}
  \int_{P_{0,t} \backslash \widetilde{D_t(\tau_1)}}
    \sum_{R \!\! \mod \Z^{(n,r)} (2 \M)} f_R(\tau) \overline{g_R(\tau)} \det(v)^{k-\frac{r}{2}+s-(n+1)}  \\
    &&
    \times
    \det(v_1)^{-s}
    \, d x'_1\, d y'_1\, d u_2\, d v_2\, d u_1\, d v_1.
\end{eqnarray*}
We write $R = \left( \begin{matrix} R_1 \\ R_2 \end{matrix} \right)$
with $R_1 \in \Z^{(n-t,r)}$ and $R_2 \in \Z^{(t,r)}$.
We take Fourier-Jacobi expansions of $f_R$ and $g_R$:
\begin{eqnarray}\label{id:FJ_fr_gr}
  f_R(\tau) &=& \sum_{N_2 \in L_t^+} f_{R,N_2}(\tau_1, z'_1)\,  e\! \left(\left( N_2 - \frac14 \M^{-1}[^t R_2]\right) \tau_2\right), \\ \notag
  g_R(\tau) &=& \sum_{N_2 \in L_t^+} g_{R,N_2}(\tau_1, z'_1)\,  e\! \left(\left( N_2 - \frac14 \M^{-1}[^t R_2]\right) \tau_2\right),
\end{eqnarray}
where we have
\begin{eqnarray*}
  f_{R,N_2}(\tau_1,z'_1)
  &=&
  \sum_{N_1 \in L_{n-t}^+, N_3 \in \Z^{(n-t,t)}}
  C_\phi(\begin{pmatrix} N_1 & \frac12 N_3 \\ \frac12 ^t N_3 & N_2 \end{pmatrix}, R)
  \\
  && \times
  e( (N_1 - \frac14 \M^{-1}[^t R_1])\tau_1 + (N_3 - \frac12 R_1 \M^{-1} {^t R_2}) {^t z'_1})
\end{eqnarray*}
and $g_{R,N_2}$ can be written similarly
by replacing $f_{R,N_2}$ (resp. $C_\phi$) by $g_{R,N_2}$ (resp. $C_\psi$).

\begin{prop}\label{prop:id:I_t}
We have the identity
\begin{equation}\label{id:I_t}
\begin{split}
   &
   I_t(\phi_\M,\psi_\M;s) \\
  =
  &\ \frac{(1+\delta_{t,n})^2}{2}
  \pi^{\frac14 t(t-1)}   \prod_{i=1}^t \Gamma\! \left(s+k-\frac{r}{2}-(n+1)+\frac{t+1}{2} - \frac{i-1}{2}\right) \\
  &
  \times 
    \int_{\Gamma_{n-t} \backslash \H_{n-t}}
    \int_{D_t(\tau_1) }
    \sum_{R_1 \!\! \mod \Z^{(n-t,r)} (2 \M)}
    \sum_{ \N \in B_{t,r}(\Z) \backslash L_{t,r}^+(\M) }
    \frac{1}{|\epsilon_{t,r}(\N)|}  \\
   &
   \times
      \det(\pi(4N_2 - \M^{-1}[^t R_2]))^{-k+\frac{r}{2}-s+(n+1)-\frac{t+1}{2}}
   \\
    &
    \times  f_{R,N_2}(\tau_1, z'_1) \overline{g_{ R, N_2}(\tau_1,z'_1 )}
      \, e(\frac{\sqrt{-1}}{2} \left(4N_2 -  \M^{-1}[^t R_2] \right) v_1^{-1}[y'_1]) 
    \det(v_1)^{k-\frac{r}{2}-(n+1)} \\
    &
    \times
    \, d x'_1\, d y'_1\, d u_1\, d v_1.
\end{split}
\end{equation}
\end{prop}
\begin{proof}
We obtain
\begin{eqnarray*}
  &&
   I_t(\phi_\M,\psi_\M;s) \\
  &=&
  \frac{1 + \delta_{t,n}}{2}
  \int_{\Gamma_{n-t} \backslash \H_{n-t}}
  \int_{P_{0,t} \backslash \widetilde{D_t(\tau_1)}}
    \sum_{R \!\! \mod \Z^{(n,r)} (2 \M)}
    \sum_{N_2, N'_2 \in L_t^+}
    f_{R,N_2}(\tau_1, z'_1) \overline{g_{R,N'_2}(\tau_1,z'_1)}\\
    &&
    \times
    e(N_2 \tau_2 - N'_2 \overline{\tau_2} - \frac{\sqrt{-1}}{2} \M^{-1}[^t R_2] v_2) 
    \det(v)^{k-\frac{r}{2}+s-(n+1)} 
    \det(v_1)^{-s} \\
   &&
   \times
    \, d x'_1\, d y'_1\, d u_2\, d v_2\, d u_1\, d v_1.
\end{eqnarray*}
Since
$\displaystyle{\int_{Sym_t(\R)\slash Sym_t(\Z)} e(N_2 \tau_2 - N'_2 \overline{\tau_2})} \, d u_2
= \delta_{N_2, N'_2}$
and since $\det(v) = \det(v_1) \det(v_2 - v_1^{-1}[y'_1])$,
we have
\begin{eqnarray*}
  &&
   I_t(\phi_\M,\psi_\M;s) \\
  &=&
  \frac{1 + \delta_{t,n}}{2}
  \int_{\Gamma_{n-t} \backslash \H_{n-t}}
  \int_{GL(t,\Z) \backslash D_t(\tau_1)' }
    \sum_{R \!\! \mod \Z^{(n,r)} (2 \M)}
    \sum_{N_2 \in L_t^+}
    f_{R,N_2}(\tau_1, z'_1) \overline{g_{R,N'_2}(\tau_1,z'_1)}\\
    &&
    \times
    e(2\sqrt{-1} N_2 v_2 - \frac{\sqrt{-1}}{2} \M^{-1}[^t R_2] v_2) 
    (\det(v_2 - v_1^{-1} [y'_1]))^{k-\frac{r}{2}+s-(n+1)} 
    \det(v_1)^{k-\frac{r}{2}-(n+1)} \\
   &&
   \times
    \, d v_2 \, d x'_1\, d y'_1\, d u_1\, d v_1,
\end{eqnarray*}
where we put $D_t(\tau_1)' := \left\{ (z'_1,  \sqrt{-1} v_2) \in \widetilde{D_t(\tau_1)} \right\}$
and where $GL(t,\Z)$ acts on $D_t(\tau_1)'$ by
$A \cdot (z'_1,\sqrt{-1}v_2) = (z'_1{^t A}, \sqrt{-1} v_2[^t A])$
for $A \in GL(t,\Z)$ and for $(z'_1, \sqrt{-1}v_2) \in D_t(\tau_1)'$.

We substitute $v_2$ by $v_2 + v_1^{-1}[y'_1]$, then
\begin{eqnarray*}
  &&
   I_t(\phi_\M,\psi_\M;s) \\
  &=&
  \frac{1 + \delta_{t,n}}{2}
  \int_{\Gamma_{n-t} \backslash \H_{n-t}}
  \int_{GL(t,\Z) \backslash D_t(\tau_1)'' }
    \sum_{R \!\! \mod \Z^{(n,r)} (2 \M)}
    \sum_{N_2 \in L_t^+}
    f_{R,N_2}(\tau_1, z'_1) \overline{g_{R,N'_2}(\tau_1,z'_1)}\\
    &&
    \times
    e(\frac{\sqrt{-1}}{2} \left(4N_2 -  \M^{-1}[^t R_2] \right) v_2) \,
    e(\frac{\sqrt{-1}}{2} \left(4N_2 -  \M^{-1}[^t R_2] \right) v_1^{-1}[y'_1]) \\
    &&
    \times
    \det(v_2)^{k-\frac{r}{2}+s-(n+1)} 
    \det(v_1)^{k-\frac{r}{2}-(n+1)} 
    \, d v_2 \, d x'_1\, d y'_1\, d u_1\, d v_1,
\end{eqnarray*}
where we put $D_t(\tau_1)''  := \{ (z'_1, \sqrt{-1}v_2) \, | \,  (z'_1,v_2) \in D_t(\tau_1) \times Sym_t^+(\R) \}$
and where we denote by $Sym_t^+(\R)$ the positive definite symmetric matrices of size $t$ which
entries are real numbers.

Since $\phi_\M$ is a Jacobi form, we obtain
\begin{eqnarray*}
C_\phi(\begin{pmatrix} N_1 & \frac12 N_3 \\ \frac12 ^t N_3 & N_2[A] \end{pmatrix}, 
        \begin{pmatrix} 1_{n-t} & \\ & ^t A \end{pmatrix} R)
&=&
C_\phi(\begin{pmatrix} N_1 & \frac12 N_3 A^{-1} \\ \frac12 {{^t A}^{-1}} ^t N_3 & N_2 \end{pmatrix}, R)
\end{eqnarray*}
for any $A \in GL(t,\Z)$.
Thus, for a fixed $R = \left( \begin{matrix} R_1 \\ R_2 \end{matrix} \right) \in \Z^{(n,r)}$, we have
\begin{eqnarray*}
  f_{\left(\begin{smallmatrix} 1_{n-t} & \\ & ^t A \end{smallmatrix}\right) R,N_2[A]}
  (\tau_1,z'_1)
  &=&
    \sum_{N_1 \in L_{n-t}^+, N_3 \in \Z^{(n-t,t)}}
  C_\phi(\begin{pmatrix} N_1 & \frac12 N_3 \\ \frac12  ^t N_3 & N_2[A] \end{pmatrix}, R)
  \\
  && \times
  e( (N_1 - \frac14 \M^{-1}[^t R_1])\tau_1 + (N_3 - \frac12 R_1 \M^{-1} {^t R_2} A) {^t z'_1 }) \\
  &=&
    \sum_{N_1 \in L_{n-t}^+, N_3 \in \Z^{(n-t,t)}}
  C_\phi(\begin{pmatrix} N_1 & \frac12 N_3 A^{-1}\\ \frac12 {{^t A}^{-1}} ^t N_3 & N_2 \end{pmatrix}, R)
  \\
  && \times
  e( (N_1 - \frac14 \M^{-1}[^t R_1])\tau_1 + (N_3 - \frac12 R_1 \M^{-1} {^t R_2}) {^t (z'_1 {^t A})}) \\
  &=&
  f_{R,N_2}(\tau_1, z'_1{^t A}).
\end{eqnarray*}

We write $\N = \begin{pmatrix} N_2 & \frac12 R_2 \\ \frac12 {^t R_2} & \M \end{pmatrix}$.
The summation 
$ \displaystyle{ \sum_{R \!\! \mod \Z^{(n,r)} (2 \M)}
    \sum_{N_2 \in L_t^+}}
$
equals to the summation
$ \displaystyle{ 
    \sum_{R_1 \!\! \mod \Z^{(n-t,r)} (2 \M)}
    \sum_{R_2 \!\! \mod \Z^{(t,r)} (2 \M)}
    \sum_{N_2 \in L_t^+}},
$
and the summation
$ \displaystyle{ 
    \sum_{R_2 \!\! \mod \Z^{(t,r)} (2 \M)}
    \sum_{N_2 \in L_t^+}}
$
equals to the summation
$
\displaystyle{ \sum_{ \N \in B_{t,r}^{\infty}(\Z) \backslash L_{t,r}^+(\M)  }}
$,
where we put
\begin{eqnarray*}
  B_{t,r}^{\infty}(\Z)
  &:=&
  \left\{ \begin{pmatrix} 1_t & y \\ 0 & 1_r \end{pmatrix} \, | \, 
  y \in \Z^{(t,r)} \right\}.
\end{eqnarray*}
We remark the isomorphism $B_{t,r}^{\infty}(\Z) \backslash B_{t,r}(\Z)  \cong GL(t,\Z)$.

Therefore we have
\begin{eqnarray*}
   &&
   I_t(\phi_\M,\psi_\M;s) \\
  &=&
  \frac{1 + \delta_{t,n}}{2}
  \int_{\Gamma_{n-t} \backslash \H_{n-t}}
  \int_{GL(t,\Z) \backslash D_t(\tau_1)'' }
    \sum_{R_1 \!\! \mod \Z^{(n-t,r)} (2 \M)}
    \sum_{R_2 \!\! \mod \Z^{(t,r)} (2 \M)}
    \sum_{N_2 \in L_t^+} f_{R,N_2}(\tau_1, z'_1) 
    \\
    &&
    \times
    \overline{g_{R,N_2}(\tau_1,z'_1)}
    e(\frac{\sqrt{-1}}{2} \left(4N_2 -  \M^{-1}[^t R_2] \right) v_2) \,
    e(\frac{\sqrt{-1}}{2} \left(4N_2 -  \M^{-1}[^t R_2] \right) v_1^{-1}[y'_1]) \\
    &&
    \times
    \det(v_2)^{k-\frac{r}{2}+s-(n+1)} 
    \det(v_1)^{k-\frac{r}{2}-(n+1)} 
    \, d v_2 \, d x'_1\, d y'_1\, d u_1\, d v_1 \\
  &=&
  \frac{1 + \delta_{t,n}}{2}
  \int_{\Gamma_{n-t} \backslash \H_{n-t}}
  \int_{GL(t,\Z) \backslash D_t(\tau_1)'' }
    \sum_{R_1 \!\! \mod \Z^{(n-t,r)} (2 \M)}
    \sum_{ \N \in B_{t,r}^{\infty}(\Z) \backslash L_{t,r}^+(\M) }
    f_{R,N_2}(\tau_1, z'_1) \\
    &&
    \times
    \overline{g_{R,N_2}(\tau_1,z'_1)}
    e(\frac{\sqrt{-1}}{2} \left(4N_2 -  \M^{-1}[^t R_2] \right) v_2) \,
    e(\frac{\sqrt{-1}}{2} \left(4N_2 -  \M^{-1}[^t R_2] \right) v_1^{-1}[y'_1]) \\
    &&
    \times
    \det(v_2)^{k-\frac{r}{2}+s-(n+1)} 
    \det(v_1)^{k-\frac{r}{2}-(n+1)} 
    \, d v_2 \, d x'_1\, d y'_1\, d u_1\, d v_1,
\end{eqnarray*}
where $\N = \begin{pmatrix} N_2 & \frac12 R_2 \\ \frac12 {^t R_2} & \M \end{pmatrix}$.
We have
\begin{eqnarray*}
   &&
   I_t(\phi_\M,\psi_\M;s) \\
 &=&
  \frac{1 + \delta_{t,n}}{2}
  \int_{\Gamma_{n-t} \backslash \H_{n-t}}
  \int_{GL(t,\Z) \backslash D_t(\tau_1)'' }
    \sum_{R_1 \!\! \mod \Z^{(n-t,r)} (2 \M)}
    \sum_{ \N \in B_{t,r}(\Z) \backslash L_{t,r}^+(\M)} \\
    &&
   \times
   \sum_{ diag(A,1_r) \in  B_{t,r}^{\infty}(\Z) \backslash B_{t,r}(\Z) \slash \epsilon_{t,r}(\N)} 
  f_{\left(\begin{smallmatrix} 1_{n-t} & \\ & ^t A \end{smallmatrix}\right) R,N_2[A]}(\tau_1, z'_1) \overline{g_{\left(\begin{smallmatrix} 1_{n-t} & \\ & ^t A \end{smallmatrix}\right) R, N_2[A]}(\tau_1,z'_1)}\\
    &&
    \times
    e(\frac{\sqrt{-1}}{2} \left(4N_2[A] -  \M^{-1}[^t R_2 A] \right) v_2) \,
    e(\frac{\sqrt{-1}}{2} \left(4N_2[A] -  \M^{-1}[^t R_2 A] \right) v_1^{-1}[y'_1]) \\
    &&
    \times
    \det(v_2)^{k-\frac{r}{2}+s-(n+1)} 
    \det(v_1)^{k-\frac{r}{2}-(n+1)} 
    \, d v_2 \, d x'_1\, d y'_1\, d u_1\, d v_1     \\
 &=&
  \frac{1 + \delta_{t,n}}{2}
  \int_{\Gamma_{n-t} \backslash \H_{n-t}}
  \int_{GL(t,\Z) \backslash D_t(\tau_1)'' }
    \sum_{R_1 \!\! \mod \Z^{(n-t,r)} (2 \M)}
    \sum_{ \N \in B_{t,r}(\Z) \backslash L_{t,r}^+(\M) }
    \sum_{ A  \in GL(t,\Z)} \frac{1}{|\epsilon_{t,r}(\N)|}
     \\
    &&
    \times
    f_{R,N_2}(\tau_1, z'_1 {^t A}) \overline{g_{ R, N_2}(\tau_1,z'_1 {^t A})}\\
    &&
    \times
    e(\frac{\sqrt{-1}}{2} \left(4N_2 -  \M^{-1}[^t R_2] \right) v_2[^t A]) \,
    e(\frac{\sqrt{-1}}{2} \left(4N_2 -  \M^{-1}[^t R_2] \right) v_1^{-1}[y'_1 {^t A}]) \\
    &&
    \times
    \det(v_2)^{k-\frac{r}{2}+s-(n+1)} 
    \det(v_1)^{k-\frac{r}{2}-(n+1)} 
    \, d v_2 \, d x'_1\, d y'_1\, d u_1\, d v_1 \\
 &=&
  \frac{(1 + \delta_{t,n})^2}{2}
  \int_{\Gamma_{n-t} \backslash \H_{n-t}}
  \int_{D_t(\tau_1) } \int_{Sym_t^+(\R)}
    \sum_{R_1 \!\! \mod \Z^{(n-t,r)} (2 \M)}
    \sum_{ \N \in B_{t,r}(\Z) \backslash L_{t,r}^+(\M)}
    \frac{1}{|\epsilon_{t,r}(\N)|} \\
    &&
    \times  f_{R,N_2}(\tau_1, z'_1) \overline{g_{ R, N_2}(\tau_1,z'_1 )}\\
    &&
    \times
    e(\frac{\sqrt{-1}}{2} \left(4N_2 -  \M^{-1}[^t R_2] \right) v_2) \,
    e(\frac{\sqrt{-1}}{2} \left(4N_2 -  \M^{-1}[^t R_2] \right) v_1^{-1}[y'_1]) \\
    &&
    \times
    \det(v_2)^{k-\frac{r}{2}+s-(n+1)} 
    \det(v_1)^{k-\frac{r}{2}-(n+1)} 
    \, d v_2 \, d x'_1\, d y'_1\, d u_1\, d v_1.
\end{eqnarray*}
As for the last identity, we remark that the set
\begin{eqnarray*}
\{ (z'_1 {^t A}, \sqrt{-1} v_2[^t A]) \, | \, 
     (z'_1,\sqrt{-1}v_2) \in GL(t,\Z) \backslash D_t(\tau_1)'', A \in GL(t,\Z) \}
\end{eqnarray*} 
covers $D_t(\tau_1)''$ twice if $t=n$ and once if $t \neq n$.

Here we have
\begin{eqnarray*}
  &&
  \int_{Sym_t^+(\R)}     e(\frac{\sqrt{-1}}{2} \left(4N_2 -  \M^{-1}[^t R_2] \right) v_2) \,
    \det(v_2)^{k-\frac{r}{2}+s-(n+1)}  \, d v_2  \\
  &=&
  \int_{Sym_t^+(\R)}     exp(- \pi\, Tr\left(4N_2 -  \M^{-1}[^t R_2] \right) v_2) \,
    \det(v_2)^{k-\frac{r}{2}+s-(n+1) + \frac{t+1}{2}} \det(v_2)^{-\frac{t+1}{2}}  \, d v_2  \\ 
  &=&
  \det(\pi(4N_2 - \M^{-1}[^t R_2]))^{-k+\frac{r}{2}-s+(n+1)-\frac{t+1}{2}} \\
  &&
  \times \int_{Sym_t^+(\R)}     exp(- Tr(v_2)) \,
    \det(v_2)^{k-\frac{r}{2}+s-(n+1) + \frac{t+1}{2}} \det(v_2)^{-\frac{t+1}{2}}  \, d v_2  \\
  &=&
  \det(\pi(4N_2 - \M^{-1}[^t R_2]))^{-k+\frac{r}{2}-s+(n+1)-\frac{t+1}{2}}
  \pi^{\frac14 t(t-1)}  \\
  &&
  \times  \prod_{i=1}^t \Gamma\! \left(s+k-\frac{r}{2}-(n+1)+\frac{t+1}{2} - \frac{i-1}{2}\right).
\end{eqnarray*}
Here, in the last identity, we used the formula shown by Maass~\cite[p.91, l.10-11]{HM}.

Thus we obtain the identity~(\ref{id:I_t}).
\end{proof}

We denote by $D_{n,r}$ the complex domain $\H_n \times \C^{(n,r)}$.
Let $\phi_\N$ and $\psi_\N$ be Fourier-Jacobi coefficients of $\phi_\M$ and $\psi_\M$
denoted in~(\ref{id:theta_decom_phi_psi}).

We write $\N = \begin{pmatrix} N_2 & \frac12 R_2 \\ \frac12 {^t R_2} & \M \end{pmatrix} \in L_{t,r}^+(\M)$.
For $z' \in \C^{(n-t,t+r)}$ we write $z' = (z'_1\ z'_2)$ with $z'_1 \in \C^{(n-t,t)}$ and $z'_2 \in \C^{(n-t,r)}$.
We have theta decompositions of $\phi_\N$ and $\psi_\N$ as follows:
\begin{eqnarray*}
  \phi_\N(\tau_1,z'_1,z'_2)
  &=&
  \sum_{R_1 \!\! \mod \Z^{(n-t,r)}(2\M)}
  f_{R, N_2}(\tau_1,z'_1)
  \vartheta_{\M,R_2,R_1}(\tau_1,z'_1,z'_2), \\
    \psi_\N(\tau_1,z'_1,z'_2)
  &=&
  \sum_{R_1 \!\! \mod \Z^{(n-t,r)}(2\M)}
  g_{R, N_2}(\tau_1,z'_1)
  \vartheta_{\M,R_2,R_1}(\tau_1,z'_1,z'_2),
\end{eqnarray*}
where we put
\begin{eqnarray}\label{id:vartheta_3}
   \vartheta_{\M,R_2,R_1}(\tau_1,z'_1,z'_2)
   &:=&
   \vartheta_{\M,R}(\tau_1, \frac12 z'_1 R_2 \M^{-1} + z'_2),
\end{eqnarray}
and where $R  = \begin{pmatrix} R_1 \\ R_2 \end{pmatrix} \in \Z^{(n,r)}$,
and where the function $\vartheta_{\M,R}$ is denoted in~(\ref{id:vartheta_2}).
Here we remark that $f_{R,N_2}$ (resp. $g_{R,N_2}$) is appeared in~(\ref{id:FJ_fr_gr})
in the Fourier-Jacobi expansion of $f_R(\tau)$ (resp. $g_R(\tau)$).
We omitted the detail of the proof of these decompositions,
since the proof is similar to~\cite[Lemma 4.1]{matrix_integer}.
In other words, we can say that the Fourier-Jacobi expansions and the theta decompositions
are compatible.

We write $\tau_1 = u_1 + i v_1$, $z'_j = x'_j + i y'_j$ $(j = 1,2)$
with matrices $u_1$, $v_1$ $\in \mbox{Sym}_{n-t}(\R)$, $x'_1$, $y'_1$ $\in \R^{(n-t,t)}$
and $x'_2$, $y'_2$ $\in \R^{(n-t,r)}$.

We need the following lemma to calculate $ D_t(\phi_\M, \psi_\M ; s)$.
\begin{lemma}\label{lem:theta_orth_2}
  We have
 \begin{eqnarray*}
  &&
   \int_{D_r(\tau_1)}
   \vartheta_{\M,R_2,R_1}(\tau_1,z'_1,z'_2)
     \overline{ \vartheta_{\M,R_2,\tilde{R}_1}(\tau_1,z'_1,z'_2)  } \\
   && \times
     e(2 \sqrt{-1} \M^{-1} v_1^{-1}[\frac12 y'_1 R_2 \M^{-1} + y'_2] )
   \, d x'_2 \, d y'_2 \\
   &=&
   \delta_{R_1, \tilde{R}_1} (\det v_1)^{\frac{r}{2}}
   \det (4\M)^{-\frac{n-t}{2}}.
  \end{eqnarray*}
\end{lemma}
\begin{proof}
By a straightforward calculation we have
\begin{eqnarray*}
  &&
   \int_{D_r(\tau_1)}
   \vartheta_{\M,R_2,R_1}(\tau_1,z'_1,z'_2)
     \overline{ \vartheta_{\M,R_2,\tilde{R}_1}(\tau_1,z'_1,z'_2)  } \\
   && \times
     e(2 \sqrt{-1} \M^{-1} v_1^{-1}[\frac12 y'_1 R_2 \M^{-1} + y'_2] )
   \, d x'_2 \, d y'_2 \\
  &=&
   \int_{D_r(\tau_1)}
   \vartheta_{\M,R}(\tau_1, \frac12 z'_1 R_2 \M^{-1} + z'_2)
     \overline{   \vartheta_{\M,\tilde{R}}(\tau_1, \frac12 z'_1 R_2 \M^{-1} + z'_2)  } \\
   && \times
     e(2 \sqrt{-1} \M^{-1} v_1^{-1}[\frac12 y'_1 R_2 \M^{-1} + y'_2] )
   \, d x'_2 \, d y'_2  \\
  &=&
   \int_{D_r(\tau_1)}
   \vartheta_{\M,R}(\tau_1, z'_2)
     \overline{   \vartheta_{\M,\tilde{R}}(\tau_1,  z'_2)  }
     e(2 \sqrt{-1} \M^{-1} v_1^{-1}[ y'_2] )
   \, d x'_2 \, d y'_2  \\
  &=&
     \delta_{R, \tilde{R}} (\det v_1)^{\frac{r}{2}}
   \det (4\M)^{-\frac{n-t}{2}} .
\end{eqnarray*}
Here, in the last identity, we used the formula in~\cite[p.211, l.19-20]{Zi}.
\end{proof}

\begin{prop}
We have
\begin{equation}\label{id:D_t}
\begin{split}
   &
    D_t(\phi_{\M},\psi_{\M}; s + k - n  + (t-r-1)/2) \\
   &=
  \frac{1 + \delta_{t,n}}{2}
 \det(2\M)^{-\frac{n-t}{2}} \det(\M)^{-s-k+n-(t-r-1)/2} 2^{-\frac{r(n-t)}{2}}
 4^{t(s+k-n+(t-r-1)/2)} \\
 &\quad
 \times
       \sum_{\mathcal{N} 
    \in B_{t,r}(\Z) \backslash L_{t,r}^+(\M) 
    } \frac{1}{
    |\epsilon_{t,r}(\mathcal{N})| \det(4N_2- \M^{-1}[^t R_2])^{s+k-n+(t-r-1)/2}} \\
  &\quad
    \times
 \int_{\Gamma_{n-t} \backslash \H_{n-t}}
 \int_{D_t(\tau_1)}
 \sum_{R_1 \!\! \mod \Z^{(n-t,r)}(2\M)}
   f_{R,N_2}(\tau_1,z'_1) 
  \overline{
   g_{R,N_2}(\tau_1,z'_1) 
  } \\
  &\quad
  \times
   \det(v_1)^{k-(n+\frac{r}{2}+1)}
 e(2i \left(N_2 - \frac14 \M^{-1}[^t R_2]\right) v_1^{-1} [y'_1]) 
  \,   d x'_1 \, d y'_1 \, d u_1 \, d v_1.
\end{split}
\end{equation}
\end{prop}
\begin{proof}
We put $D_r(\tau_1) = \C^{(n-t,r)} / (\tau_1 \Z^{(n-t,r)} + \Z^{(n-t,r)})$.
We have
\begin{eqnarray*}
 \langle \phi_\N, \psi_\N \rangle
 &=&
 \int_{\Gamma^J_{n-t,t+r} \backslash D_{n-t,t+r}}
 \phi_\N(\tau_1,z'_1,z'_2) \overline{\psi_\N(\tau_1,z'_1,z'_2)}
 e(2i \N v_1^{-1}[(y'_1\ y'_2)]) \\
 &&
 \times
 \det(v_1)^{k-(n+r+1)}
 \, d u_1 \, d v_1 \, d x'_1 \, d y'_1 \, d x'_2 \, d y'_2 \\
 &=&
 \frac{1 + \delta_{t,n}}{2}
 \int_{\Gamma_{n-t} \backslash \H_{n-t}}
 \int_{D_t(\tau_1)}
 \int_{D_r(\tau_1)}
 \phi_\N(\tau_1,z'_1,z'_2) \overline{\psi_\N(\tau_1,z'_1,z'_2)}
 e(2i \N v_1^{-1}[(y'_1\ y'_2)]) \\
 &&
 \times
 \det(v_1)^{k-(n+r+1)}
  \, d x'_2 \, d y'_2 \,   d x'_1 \, d y'_1 \, d u_1 \, d v_1 \\
 &=&
 \frac{1 + \delta_{t,n}}{2}
 \int_{\Gamma_{n-t} \backslash \H_{n-t}}
 \int_{D_t(\tau_1)}
 \int_{D_r(\tau_1)}
 \sum_{R_1 \!\! \mod \Z^{(n-t,r)}(2\M)}
   f_{R,N_2}(\tau_1,z'_1) \vartheta_{\M,R_2,R_1}(\tau_1,z'_1,z'_2)
 \\
 &&
 \times
 \sum_{\tilde{R}_1 \!\! \mod \Z^{(n-t,r)}(2\M)}
  \overline{
   g_{\tilde{R},N_2}(\tau_1,z'_1) \vartheta_{\M,R_2,\tilde{R}_1}(\tau_1,z'_1,z'_2)  
  }
 e(2i \N v_1^{-1}[(y'_1\ y'_2)]) \\
 &&
 \times
 \det(v_1)^{k-(n+r+1)}
  \, d x'_2 \, d y'_2 \,   d x'_1 \, d y'_1 \, d u_1 \, d v_1,
\end{eqnarray*}
where we write $R = \begin{pmatrix} R_1 \\ R_2 \end{pmatrix}$
and $\tilde{R} = \begin{pmatrix} \tilde{R}_1 \\ R_2 \end{pmatrix}$.
Since
\begin{eqnarray*}
\N = \begin{pmatrix} N_2 & \frac12 R_2 \\ \frac12 {^t R_2} & \M \end{pmatrix}
= \begin{pmatrix} N_2 - \frac14 \M^{-1}[^t R_2] & 0 \\ 0 & \M \end{pmatrix}
  \left[
    \begin{pmatrix} 1_{t} &  \\ \frac12 \M^{-1} {^t R_2} & 1_r\end{pmatrix}
  \right],
\end{eqnarray*} 
we obtain
\begin{eqnarray*}
 \langle \phi_\N, \psi_\N \rangle
 &=&
  \frac{1 + \delta_{t,n}}{2}
 \int_{\Gamma_{n-t} \backslash \H_{n-t}}
 \int_{D_t(\tau_1)}
 \int_{D_r(\tau_1)}
 \sum_{R_1 \!\! \mod \Z^{(n-t,r)}(2\M)}
   f_{R,N_2}(\tau_1,z'_1) \vartheta_{\M,R_2,R_1}(\tau_1,z'_1,z'_2)
 \\
 &&
 \times
 \sum_{\tilde{R}_1 \!\! \mod \Z^{(n-t,r)}(2\M)}
  \overline{
   g_{\tilde{R},N_2}(\tau_1,z'_1) \vartheta_{\M,R_2,\tilde{R}_1}(\tau_1,z'_1,z'_2)  
  } \\
  &&
  \times
 e(2i \left(N_2 - \frac14 \M^{-1}[^t R_2]\right) v_1^{-1} [y'_1])\ 
 e(2i \M v_1^{-1} \left[\frac12 y'_1 R_2 \M^{-1} + y'_2\right]) \\
 &&
 \times
 \det(v_1)^{k-(n+r+1)}
  \, d x'_2 \, d y'_2 \,   d x'_1 \, d y'_1 \, d u_1 \, d v_1 .
\end{eqnarray*}
Due to Lemma~\ref{lem:theta_orth_2} we obtain
\begin{eqnarray*}
 \langle \phi_\N, \psi_\N \rangle
 &=&
 \frac{1 + \delta_{t,n}}{2}
 \int_{\Gamma_{n-t} \backslash \H_{n-t}}
 \int_{D_t(\tau_1)}
 \sum_{R_1 \!\! \mod \Z^{(n-t,r)}(2\M)}
   f_{R,N_2}(\tau_1,z'_1) 
  \overline{
   g_{R,N_2}(\tau_1,z'_1) 
  } \\
  &&
  \times
 e(2i \left(N_2 - \frac14 \M^{-1}[^t R_2]\right) v_1^{-1} [y'_1]) \\
 &&
 \times
 \det(v_1)^{k-(n+\frac{r}{2}+1)}
 \det(2\M)^{-\frac{n-t}{2}} 2^{-\frac{r(n-t)}{2}}
  \,   d x'_1 \, d y'_1 \, d u_1 \, d v_1.
\end{eqnarray*}
Therefore we have
\begin{eqnarray*}
   &&
    D_t(\phi_{\M},\psi_{\M}; s + k - n  + (t-r-1)/2) \\
   &=&
 \frac{1 + \delta_{t,n}}{2}
 \det(2\M)^{-\frac{n-t}{2}} 2^{-\frac{r(n-t)}{2}} \\
 &&
 \times
       \sum_{\mathcal{N} 
    \in B_{t,r}(\Z) \backslash L_{t,r}^+(\M) 
    } \frac{1}{
    |\epsilon_{t,r}(\mathcal{N})| \det(\mathcal{N})^{s+k-n+(t-r-1)/2}} \\
    &&
    \times
 \int_{\Gamma_{n-t} \backslash \H_{n-t}}
 \int_{D_t(\tau_1)}
 \sum_{R_1 \!\! \mod \Z^{(n-t,r)}(2\M)}
   f_{R,N_2}(\tau_1,z'_1) 
  \overline{
   g_{R,N_2}(\tau_1,z'_1) 
  } \\
  &&
  \times
   \det(v_1)^{k-(n+\frac{r}{2}+1)}
 e(2i \left(N_2 - \frac14 \M^{-1}[^t R_2]\right) v_1^{-1} [y'_1]) 
  \,   d x'_1 \, d y'_1 \, d u_1 \, d v_1 \\
   &=&
 \frac{1 + \delta_{t,n}}{2}
 \det(2\M)^{-\frac{n-t}{2}} \det(\M)^{-s-k+n-(t-r-1)/2} 2^{-\frac{r(n-t)}{2}}
 4^{t(s+k-n+(t-r-1)/2)} \\
 &&
 \times
       \sum_{\mathcal{N} 
    \in B_{t,r}(\Z) \backslash L_{t,r}^+(\M) 
    } \frac{1}{
    |\epsilon_{t,r}(\mathcal{N})| \det(4N_2- \M^{-1}[^t R_2])^{s+k-n+(t-r-1)/2}} \\
    &&
    \times
 \int_{\Gamma_{n-t} \backslash \H_{n-t}}
 \int_{D_t(\tau_1)}
 \sum_{R_1 \!\! \mod \Z^{(n-t,r)}(2\M)}
   f_{R,N_2}(\tau_1,z'_1) 
  \overline{
   g_{R,N_2}(\tau_1,z'_1) 
  } \\
  &&
  \times
   \det(v_1)^{k-(n+\frac{r}{2}+1)}
 e(2i \left(N_2 - \frac14 \M^{-1}[^t R_2]\right) v_1^{-1} [y'_1]) 
  \,   d x'_1 \, d y'_1 \, d u_1 \, d v_1.
\end{eqnarray*}
We conclude this proposition.
\end{proof}

By comparing the identity~(\ref{id:I_t}) with the identity~(\ref{id:D_t}),
we obtain Proposition~\ref{prop:int_D}.


\vspace{1cm}

\noindent
Department of Mathematics, Joetsu University of Education,\\
1 Yamayashikimachi, Joetsu, Niigata 943-8512, JAPAN\\
e-mail hayasida@juen.ac.jp

\end{document}